\documentclass{amsart}

\usepackage{amsmath, amsthm, amssymb}
\usepackage{graphicx}
\usepackage{tikz, tikz-cd}
\usetikzlibrary{arrows, matrix}
\usepackage{ytableau}
\usepackage{subfigure}
\usepackage[hidelinks]{hyperref}

\newcommand{\sg}{\sigma}

\def\inv{\operatorname{inv}}

\def\maj{\operatorname{maj}}
\def\minimaj{\operatorname{minimaj}}

\newcommand{\qbinom}[2] { {#1 \brack #2}_{q} }

\newcommand{\qint}[1] { [#1]_{q} }
\newcommand{\qtint}[1] { [#1]_{q,t} }

\newcommand{\osp}[2] {\mathcal{OP}_{#1, #2}}

\def\S{\mathfrak{S}}

\def\multiset#1#2{\ensuremath{\left(\kern-.3em\left(\genfrac{}{}{0pt}{}{#1}{#2}\right)\kern-.3em\right)}}

\def\Dinv{\operatorname{Dinv}}
\def\dinv{\operatorname{dinv}}
\def\area{\operatorname{area}}

\newcommand{\ldyck}{\mathcal{LD}}
\newcommand{\pldyck}{\mathcal{LD}^{\operatorname{Part}}}
\newcommand{\dense}{\operatorname{Dense}}

\def\lab{\el}
\def\touch{\operatorname{touch}}

\newcommand{\dyck} { \mathcal{D}}

\def\rise{\operatorname{Rise}}
\def\fall{\operatorname{Fall}}
\def\val{\operatorname{Val}}

\def\risepoly{\rise}
\def\valpoly{\val}
\def\comp{\operatorname{comp}}

\newcommand{\diag} { \operatorname{Diag}}
\def\height{h}
\def\col{c}

\def\hinv{\operatorname{hdinv}}

\def\winv{\operatorname{wdinv}}

\def\dinvi{d}
\def\areai{a}
\newcommand{\stack}{\operatorname{Stack}}

\def\llt{LLT}
\def\word{w}
\def\cat{\operatorname{Cat}}
\def\catmod{\cat^{\prime}}
\def\bi{b}
\def\cpoly{C}
\def\el{\ell}

\def\touch{\operatorname{touch}}

\usepackage{hyperref}

\newtheorem{lemma}{Lemma}[section]
\newtheorem{prop}{Proposition}[section]
\newtheorem{cor}{Corollary}[section]
\newtheorem{thm}{Theorem}[section]
\newtheorem{conj}{Conjecture}[section]
\newtheorem{prob}{Problem}[section]

\author{J.\ Haglund}
\address{Department of Mathematics, University of Pennsylvania \\ Philadelphia, PA 19104}
\email{\texttt{jhaglund@math.upenn.edu}}
\thanks{The first author was partially supported by NSF grant DMS-1200296.}

\author{J.\ B.\ Remmel}
\address{Department of Mathematics, UC San Diego \\ La Jolla, CA 92093}
\email{\texttt{jremmel@math.ucsd.edu}}

\author{A.\ T.\ Wilson}
\address{Department of Mathematics, University of Pennsylvania \\ Philadelphia, PA 19104}
\email{\texttt{andwils@math.upenn.edu}}
\thanks{The third author was supported by a DoD National Defense Science and Engineering Graduate Fellowship and an NSF Mathematical Sciences Postdoctoral Research Fellowship. }

\title{The Delta Conjecture}

\begin{document}
\begin{abstract}
We conjecture two combinatorial interpretations for the symmetric function $\Delta_{e_k} e_n$, where $\Delta_f$ is an eigenoperator for the modified Macdonald polynomials defined by Bergeron, Garsia, Haiman, and Tesler. Both interpretations can be seen as generalizations of the Shuffle Conjecture of Haglund, Haiman, Remmel, Loehr, and Ulyanov, which was proved recently by Carlsson and Mellit. We show how previous work of the third author on Tesler matrices and ordered set partitions can be used to verify several cases of our conjectures. Furthermore, we use a reciprocity identity and LLT polynomials to prove another case. Finally, we show how our conjectures inspire 4-variable generalizations of the Catalan numbers, extending work of Garsia, Haiman, and the first author.
\end{abstract}
\maketitle

\tableofcontents

\section{Introduction}
\label{sec:intro}
While working towards a proof of the Schur positivity of Macdonald polynomials, Garsia and Haiman discovered the module of diagonal harmonics, an $\S_n$-module that captures many of the properties of Macdonald polynomials. In \cite{delta}, Haiman proved that the Frobenius characteristic of the diagonal harmonics could be written as $\nabla e_n$ or $\Delta_{e_n} e_n$ for certain symmetric function operators $\nabla$ and $\Delta_f$ which are eigenoperators of Macdonald polynomials. Building on this work, Haiman, Loehr, Ulyanov, and the first two authors proposed a connection between $\nabla e_n$ and parking functions which has come to be known as the Shuffle Conjecture \cite{shuffle}. The goal of this paper is to state and support two versions of a generalization of the Shuffle Conjecture in which $\Delta_{e_n} e_n$ is replaced by $\Delta_{e_k} e_n$ for an integer $1 \leq k \leq n$. We will also see how our generalizations tie together a wide variety of algebraic and combinatorial objects, such as parking functions, ordered set partitions, generalizations of Tesler matrices, and LLT polynomials. In this section, we provide the necessary notation and then state our conjecture.

\begin{figure}
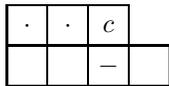

\begin{displaymath}
\begin{ytableau}
\cdot 	& \cdot	& c  \\
\phantom{\cdot}	&  \phantom{\cdot} 	&  -	&  \phantom{\cdot}
\end{ytableau}
\end{displaymath}
\caption{This is the Young diagram (in French notation) of the partition $(4,3)$. The cell $c$ has $a^{\prime}(c) = 2$ (represented by dots) and $\el^{\prime}(c) = 1$ (represented by dashes).}
\label{fig:arm-leg}
\end{figure}

Let $\Lambda$ denote the ring of symmetric functions with coefficients in $\mathbb{Q}(q,t)$. The sets $\{e_{\mu} : \mu \vdash n\}$ and $\{\tilde{H}_{\mu} : \mu \vdash n\}$ are the elementary and {(modified) Macdonald} symmetric function bases for $\Lambda^{(n)}$, the elements of $\Lambda$ that are homogeneous of degree $n$. Given a partition $\mu \vdash n$ and a cell $c$ in the Young diagram of $\mu$ (drawn in French notation) we set $a^{\prime}(c)$ and $\el^{\prime}(c)$ to be the number of cells in $\mu$ that are strictly to the left and strictly below $c$ in $\mu$, respectively. We define 
\begin{align}
&B_{\mu}(q,t) = \sum_{c \in \mu} q^{a^{\prime}(c)} t^{\el^{\prime}(c)} \ \ \ \ \ \text{and} \ \ \ \  \ T_{\mu}(q,t) = \prod_{c \in \mu} q^{a^{\prime}(c)} t^{\el^{\prime}(c)} .
\end{align} 
Given any symmetric function $f \in \Lambda$, we define operators $\Delta_f$ and $\Delta^{\prime}_f$ on $\Lambda$ by their action on the Macdonald basis:
\begin{align}
&\Delta_f \tilde{H}_{\mu} = f[B_{\mu}(q,t)] \tilde{H}_{\mu} \ \ \ \ \text{and} \ \ \ \ 
\Delta^{\prime}_f  \tilde{H}_{\mu} = f[B_{\mu}(q,t)-1] \tilde{H}_{\mu}.
\end{align}
Here, we have used the notation that, for a symmetric function $f$ and a sum $A = a_1 + \ldots + a_N$ of monic monomials, $f[A]$ is equal to the specialization of $f$ at $x_1=a_1, \ldots, x_N=a_N$, where the remaining variables are set equal to zero. We also set $\nabla = \Delta_{e_n}$ as an operator on $\Lambda^{(n)}$. 

Our goal is to conjecture combinatorial interpretations for $\Delta_{e_k} e_n$ and $\Delta^{\prime}_{e_k} e_n$. Note that, by definition, for any $1 \leq k \leq n$
\begin{align}
\Delta_{e_k} e_n = \Delta^{\prime}_{e_k + e_{k-1}} e_n = \Delta^{\prime}_{e_k} e_n + \Delta^{\prime}_{e_{k-1}} e_n .
\end{align}
Furthermore, for any $k > n$, $\Delta_{e_k} e_n = \Delta^{\prime}_{e_{k-1}} e_n = 0$. Therefore $\Delta_{e_n} e_n = \Delta^{\prime}_{e_{n-1}} e_n$. 

To state our conjectures, we consider parking functions as labeled Dyck paths. A \emph{Dyck path} of order $n$ is a lattice path from $(0,0)$ to $(n,n)$ consisting of north and east steps that remains weakly above the line $y=x$, which is sometimes called the diagonal, main diagonal, or 0-diagonal. To obtain a \emph{labeled Dyck path}, we label the north steps of a Dyck path with (not necessarily unique) positive integers such that the labels strictly increase while ascending each column. We denote the Dyck paths and labeled Dyck paths of order $n$ by $\dyck_n$ and $\ldyck_n$, respectively. Often, we will use the notation $D(P)$ to denote the underlying Dyck path of a labeled Dyck path $P$. Labeled Dyck paths are sometimes called (word) parking functions, since there is a classical bijection between the two classes of objects.

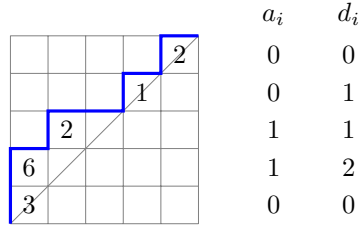
\begin{figure}
\begin{tikzpicture}
\draw[step=0.5cm, gray, very thin] (0.999,0) grid (3.5,2.5);
\draw[ gray, very thin] (1,0) -- (3.5,2.5);
\draw[blue, very thick] (1,0) -- (1,1) -- (1.5,1) -- (1.5,1.5) -- (2.5,1.5) -- (2.5, 2) -- (3, 2) -- (3,2.5) -- (3.5,2.5);
 
\node at (1.25,0.25) {3};
\node at (1.25,0.75) {6};
\node at (1.75,1.25) {2};
\node at (2.75,1.75) {1};
\node at (3.25,2.25) {2};

\node at (4.5,2.75) {$\areai_i$};
\node at (4.5,0.25) {0};
\node at (4.5,0.75) {1}; 
\node at (4.5,1.25) {1};
\node at (4.5,1.75) {0};
\node at (4.5,2.25) {0};

\node at (5.5,2.795) {$\dinvi_i$};
\node at (5.5,0.25) {0};
\node at (5.5,0.75) {2};
\node at (5.5,1.25) {1};
\node at (5.5,1.75) {1};
\node at (5.5,2.25) {0};

\end{tikzpicture}
\caption{A sample labeled Dyck path $P \in \ldyck_5$ with $\area(P) = 2$, $\dinv(P) = 4$, $\comp(P) = \{1,2,1,1\}$, and $\val(P) = \{4,5\}$. }
\label{fig:pf}
\end{figure}

Given a Dyck path $D \in \dyck_n$, we number the rows of $D$ with $1,2,\ldots,n$ from bottom to top. Then, for each row $i$, we set the \emph{area} of the row $i$, written $\areai_i(D)$, to be the number of full squares between $P$ and the diagonal. A labeled Dyck path $P$ inherits the values $\areai_i(P)$ from its underlying Dyck path $D(P)$. We also set 
\begin{align}
\dinvi_{i}(P) &= | \{i < j \leq n : \areai_i(P) = \areai_j(P), \lab_i(P) < \lab_j(P)  \} |  \\
\nonumber	&+ | \{i < j \leq n: \areai_i(P) = \areai_j(P) + 1, \lab_i(P) > \lab_j(P) \} | .
\end{align}
where $\lab_i(P)$ is the label in the $i$th row of $P$. These are the primary and secondary diagonal inversions beginning in row $i$, respectively. The area and dinv statistics are defined by $\area(P) = \sum_{i=1}^{n} \areai_i(P)$ and $\dinv(P) = \sum_{i=1}^{n} \dinvi_i(P)$. 

The \emph{contractible valleys} of $P$ are 
\begin{align}
\val(P) &= \{2 \leq i \leq n : \areai_i(P) < \areai_{i-1}(P) \} \\
	&\cup \{2 \leq i \leq n : \areai_i(P) = \areai_{i-1}(P), \lab_i(P) > \lab_{i-1}(P) \} .
\end{align}
Visually, these are the rows $i$ that are immediately preceded by an east step and, if we were to remove this east step and shift everything beyond it one step to the west, the resulting labeled path would still have increasing labels in its columns. For the parking function depicted in Figure \ref{fig:pf}, 3 is not a contractible valley because removing the east step that starts row 3 would result in a 2 above a 6, so that column's labels would no longer increase from bottom to top. However, rows 4 and 5 are contractible valleys. Finally, by $x^P$ we mean the monomial $\prod_{i=1}^{n} x_{\lab_i(P)}$ and for any polynomial $f(z)$ we use $f(z)|_{z^k}$ to denote the coefficient of $z^k$ in $f$.

With these definitions in hand, we can state our main conjecture, which we call the Delta Conjecture. Sometimes we will refer to \eqref{rise-intro} as the Rise Version and \eqref{valley-intro} as the Valley Version of the Delta Conjecture. 

\begin{conj}[Delta Conjecture]
\label{conj:delta}
For any integers $n > k \geq 0$,
\begin{align}
\label{rise-intro}
 \Delta^{\prime}_{e_k} e_n &= \left. \sum_{P \in \ldyck_{n}} q^{\dinv(P)} t^{\area(P)} \prod_{i : \, \areai_i(P) > \areai_{i-1}(P)} \left( 1 + z / t^{\areai_i(P)} \right) x^P  \right|_{z^{n-k-1}}  \\
 \label{valley-intro}
&= \left. \sum_{P \in \ldyck_{n}} q^{\dinv(P)} t^{\area(P)} \prod_{i \in \val(P)} \left(1 + z / q^{\dinvi_{i}(P) + 1} \right) x^P \right|_{z^{n-k-1}} . 
\end{align}
Equivalently, we can replace the left-hand side with $\Delta_{e_k} e_n$ for integers $n \geq k \geq 0$, multiply both right-hand sides by $(1+z)$, and then take the coefficient  of $z^{n-k}$. 
\end{conj}

Most of the remainder of the paper is devoted to this conjecture. Several cases of the Delta Conjecture have been proved in existing work, usually under different guises. We summarize the current status of this progress in Figure \ref{fig:status}. We establish the necessary background in Section \ref{sec:background}. In Section \ref{sec:comb}, we define several classes of combinatorial objects and use these objects to give alternate formulations of the Delta Conjecture. Section \ref{sec:osp} connects previous work of the second and third authors on ordered set partitions and Tesler matrices in \cite{rw, tesler-wilson} to cases of the Delta Conjecture. Section \ref{sec:t=1/q} provides a plethystic formula for $\Delta_{e_k} e_n$ at $t=1/q$ and uses a result of Garsia, Leven, Wallach, and Xin \cite{rational-t=1/q} to prove Schur positivity in this case. In Section \ref{sec:k=1}, we use a reciprocity identity and LLT polynomials to prove the $k=1$ case of \eqref{rise-intro}. Section \ref{sec:extensions} contains a variation of the Delta Conjecture that involves new 4-variable Catalan polynomials. Finally, we use Section \ref{sec:final} to outline some of the major open problems related to the Delta Conjecture. In particular, Carlsson and Mellit \cite{carlsson-mellit} have recently announced a proof of the Shuffle Conjecture, which is equivalent to the $k=n-1$ case of the Delta Conjecture. We briefly discuss how their ideas might be used in our setting.

\begin{figure}
\begin{tabular}{c | c | c | c }
Conditions &  LHS of \eqref{rise-intro} & RHS of \eqref{rise-intro} & RHS of \eqref{valley-intro}  \\ \hline
$\langle \cdot, p_{1^n} \rangle$ at $q=0$		& \cite{tesler-wilson}		&\cite{rw}	&\cite{rw}	  \\ 
$\langle \cdot, p_{1^n} \rangle$ at $t=0$		& \cite{tesler-wilson}		&\cite{rw}	&\cite{rhoades-minimaj}	 	\\ 
$\langle \cdot, p_{1^n} \rangle$ at $q=1$		& \cite{tesler-wilson}		& \cite{tesler-wilson} & ? \\
$\langle \cdot, e_{n-d}h_{d}\rangle$, $t=1/q$	& \cite{wilson-thesis}	&\cite{wilson-thesis}	&\cite{wilson-thesis}	 \\ 
$\langle \cdot, h_{n-d}h_{d}\rangle$, $t=1/q$	& \cite{wilson-thesis}	&\cite{wilson-thesis}	&\cite{wilson-thesis} \\
$k=1$		& Section \ref{sec:k=1}		& Section \ref{sec:k=1}	& ?		\\ 
\end{tabular}
\caption{This table summarizes the progress of work on Conjecture \ref{conj:delta}, also known as the Delta Conjecture. The connections between the citations and the Delta Conjecture are explained in Section \ref{sec:osp}. Question marks indicate cases which have not been proved. }
\label{fig:status}
\end{figure}

\section{Background}
\label{sec:background}
In this section, we fix the notation that we did not establish in the introduction. A partition $\lambda \vdash n$ is a weakly decreasing sequence of positive integers of length $\ell(\lambda)$ whose sum is $n$. The sets $\{e_{\lambda} : \lambda \vdash n\}$, $\{m_{\lambda} : \lambda \vdash n\}$, $\{h_{\lambda} : \lambda \vdash n\}$, $\{p_{\lambda} : \lambda \vdash n\}$, $\{s_{\lambda} : \lambda \vdash n\}$, and $\{\tilde{H}_{\lambda} : \lambda \vdash n\}$ are the elementary, monomial, homogeneous, power sum, Schur, and (modified) Macdonald bases for $\Lambda^{(n)}$, the symmetric functions that are homogeneous of degree $n$. A symmetric function is a formal power series in variables $x_1, x_2, x_3, \ldots$ that are invariant under permuting the indices of the variables. We will also use the classical Hall inner product and the involution $\omega$ on $\Lambda$. More information about these topics can be found in \cite{ec2, macdonald, hhl}. 

We will find that the concept of plethysm is quite valuable, especially in Section \ref{sec:t=1/q}. Given a power series $E$ in the variables $q$, $t$ and $x_1, x_2, x_3, \ldots$, we consider $E$ as a sum of monomials. Then, for any symmetric function $f \in \Lambda$, we define the plethysm $p_k[E]$ to be the sum of all the monomials in $E$ raised to the $k$th power. Extending by multiplication, this defines $p_{\lambda}[E]$ for any partition $\lambda$. Finally, for any symmetric function $f$ we compute $f[E]$ by expanding $f$ into the power sum basis and then replacing each $p_{\lambda}$ with $p_{\lambda}[E]$. Sometimes we will use $X$ to denote the sum $x_1 + x_2 + x_3 + \ldots .$ With this notation, we can state a useful identity that is sometimes called Cauchy's Formula: for any bases $\{a_{\lambda} : \lambda \vdash n\}$ and $\{b_{\lambda} : \lambda \vdash n\}$ that are dual with respect to the Hall inner product and two sums $X$ and $Y$,
\begin{align}
e_{n}[XY] &= \sum_{\lambda \vdash n} \omega\left( a_{\lambda}[X]\right) b_{\lambda}[Y] .
\end{align}

The ring of quasisymmetric functions consists of the formal power series in variables $x_1, x_2, x_3, \ldots$ that are invariant under any permutation of the indices that preserves the order of the indices. We will only use the monomial basis $\{M_{\alpha} : \alpha \vDash n\}$ for the quasisymmetric functions that are homogeneous of degree $n$, where $\alpha$ is a composition (i.e.\ vector of positive integers) whose sum is $n$. The book \cite{ec2} contains more information on quasisymmetric functions for the curious reader.

Finally, we use the standard notation for $q$- and $q,t$- integers and binomial coefficients. For integers $n \geq k \geq 0$, 
\begin{align}
\qint{n} &= \sum_{i=0}^{n-1} q^i  &\qtint{n} = \sum_{i=0}^{n-1} q^i t^{n-i-1} \\
\qint{n}! &= \prod_{i=1}^{n} \qint{i} &\qbinom{n}{k} = \frac{\qint{n}!}{\qint{k}! \qint{n-k}!}
\end{align}
with the convention that $\qint{0} = \qtint{0} = 1$.

\section{Alternate combinatorial formulations}
\label{sec:comb}
Set $\risepoly_{n,k}(x;q,t)$ and $\valpoly_{n,k}(x;q,t)$ to be the right-hand sides of \eqref{rise-intro} and \eqref{valley-intro}, respectively. In this section, we define several classes of combinatorial objects and give statistics on these objects which lead to alternate formulas for $\risepoly_{n,k}(x;q,t)$ and $\valpoly_{n,k}(x;q,t)$ .  These formulations make it easier to approach special cases of the Delta Conjecture in Sections \ref{sec:osp} and \ref{sec:k=1}. We also hope that they may be useful in future work on the Delta Conjecture.

\subsection{Decorated labeled Dyck paths}
\label{ssec:decorating}
We begin by decorating labeled Dyck paths. Specifically, given $P \in \ldyck_n$, let the \emph{double rises} of $P$ be the set 
\begin{align}
\rise(P) &= \{2 \leq i \leq n : \areai_i(P) > \areai_{i-1}(P) \} .
\end{align}
These are the rows whose north step is immediately preceded by another north step. Similarly, we define the \emph{double falls} of $P$, written $\fall(P)$, to be the columns of $P$ whose east step is immediately followed by another east step. Then we can define the double rise-decorated, double fall-decorated, and contractible valley-decorated labeled Dyck paths, respectively, as follows:
\begin{align}
\ldyck^{\rise}_{n,k} &= \{ (P, R) : P \in \ldyck_n, R \subseteq \rise(P), |R| = k \} \\
\ldyck^{\fall}_{n,k} &= \{ (P, F) : P \in \ldyck_n, F \subseteq \fall(P), |F| = k \} \\
\ldyck^{\val}_{n,k} &= \{ (P, V) : P \in \ldyck_n, V \subseteq \val(P), |V| = k \} .
\end{align}
There is a trivial bijection between $\ldyck^{\rise}_{n,k}$ and $\ldyck^{\fall}_{n,k}$; namely, given a row $i \in R$ with $\areai_i(P) = a$, send $i$ to the column which contains the first east step north of $i$ that is $a$ lattice steps away from the diagonal. This is equivalent to matching open and closed parentheses in Dyck words. We will give a bijection connecting each of these sets to $\ldyck^{\val}_{n,k}$ later in this section. For now, we define statistics on these objects as follows. For $P \in \ldyck_n$, $R \subseteq \rise(P)$, $F \subseteq \fall(P)$, and $V \subseteq \val(P)$, we set
\begin{align}
\area^{-}((P, R)) &= \sum_{i \in \{1,2,\ldots,n\} \setminus R} \areai_i(P), \\
\area^{-}((P, F)) &= \sum_{i \in \{1,2,\ldots,n\} \setminus F} \col_i(P), \text { and} \\
\dinv^{-}((P, V)) &= \sum_{i \in \{1,2,\ldots,n\} \setminus V} \dinvi_i(P) - |V|,
\end{align}
where $\col_i(P)$ is the number of full squares between $P$ and the diagonal in the $i$th column.

It is not immediately clear from its definition that $\dinv^{-}((P,V))$ is always nonnegative. To see this, consider a (contractible) valley $v$ of a labeled Dyck path $P \in \ldyck_n$. We will show that there is always at least one diagonal inversion of the form $(i,v)$ for $i < v$ with $i \notin \val(P)$. By definition, we must have $v > 1$. If $\areai_{v-1} = \areai_{v}$, then by the definition of contractible valleys $(v-1,v)$ is a diagonal inversion. Now assume that $\areai_{v-1} > \areai_v$. Then there must be a row $j < v$ with $\areai_j = \areai_v$ such that $j+1 \in \rise(P)$. Choose the smallest such $j$. If $j \in \val(P)$, choose $i$ to be as large as possible so that each of $i+1, i+2, \ldots, j \in \val{P}$. By the definition of $i$ and by the choice of $j$, $i$ cannot be a valley. Since $j+1 \in \rise(P)$, $j+1 \notin \val(P)$. We claim that at least one of $(i,v)$ and $(j+1,v)$ is a diagonal inversion. $(i,v)$ is a primary diagonal inversion unless $\lab_i(P) \geq \lab_v(P)$; in that case, $\lab_{j+1}(P) > \lab_{j}(P) > \lab_{i}(P)$, so $(j+1,v)$ is a secondary diagonal inversion.

The following identities follow directly from the definitions given above. They give alternate expressions for the right-hand sides of Conjecture \ref{conj:delta} and, thanks to the argument above, show that the powers of $q$ and $t$ in $\eqref{valley-intro}$ are always nonnegative.

\begin{prop}
\label{prop:decorating}
For integers $n > k \geq 0$,
\begin{align}
\risepoly_{n,k}(x;q,t)  &= \sum_{(P, R) \in \ldyck^{\rise}_{n,n-k-1}} q^{\dinv(P)} t^{\area^{-}((P,R))} x^P \\
&= \sum_{(P, F) \in \ldyck^{\fall}_{n,n-k-1}} q^{\dinv(P)} t^{\area^{-}((P,F))} x^P . \\
\valpoly_{n,k}(x;q,t)&= \sum_{(P, V) \in \ldyck^{\rise}_{n,n-k-1}} q^{\dinv^{-}((P,V))} t^{\area(P)} x^P. 
\end{align}
\end{prop}

\subsection{Leaning stacks}
\label{ssec:stacks}
In this section, we define a class of objects which will allow us to state the two forms of the Delta Conjecture on a single set of objects. We consider what we call \emph{leaning stacks}. A leaning stack is a sequence of $n$ unit lattice square boxes, each of which is either just northeast of the box below it or directly north of the box below it. We denote the set of leaning stacks with $n$ boxes, $k$ of which are diagonally above the square blow them, by $\stack_{n,k}$. 

For a fixed leaning stack $S \in \stack_{n,k}$, the \emph{labeled Dyck paths} with respect to $S$, denoted $\ldyck(S)$, are the lattice paths consisting of north and east steps from $(0,0)$ to $(k+1,n)$ that remain weakly to the left of the left border of $S$ and which are labeled according to the same rules as stated in Section \ref{sec:intro}. We denote the unlabeled versions of these objects by $\dyck(S)$. We set $\ldyck^{\stack}_{n,k} = \cup_{S \in \stack_{n,k}} \ldyck(S)$.

We claim that $\ldyck^{\stack}_{n,k}$ is in bijection with each of $\ldyck^{\rise}_{n,n-k-1}$, $\ldyck^{\fall}_{n,n-k-1}$, and $\ldyck^{\val}_{n,n-k-1}$. Furthermore, we can translate the statistics from these sets of objects to $\ldyck^{\stack}_{n,k}$. Given $P \in \ldyck(S)$ with leaning stack $S \in \stack_{n,k}$, for each row of $P$ set $\areai_i(P)$ to be the number of squares between $P$ and $S$ and $\height_i(P)$ to be the number of squares strictly below the square just to the right of the north step in row $i$ and weakly above the bottom square of $S$ in the same column. Then $\area(P) = \sum_{i=1}^{n} \areai_i(P)$ is simply the number of squares between $P$ and $S$. (Note that this is not equal to $\sum_{i=1}^{n} \height_i(P)$.) Set $\diag(S)$ to be the rows of $S$ which are diagonally above the square below them along with row 1. Then we can define
\begin{align}
\winv(P) &= | \{1 \leq i < j \leq n : i \in \diag(S), \areai_i(P) = \areai_j(P), \lab_i(P) < \lab_j(P) \} | \\
\nonumber
&+ |\{1 \leq i < j \leq n : i \in \diag(S), \areai_i(P) = \areai_j(P)+1, \lab_i(P) > \lab_j(P) \} | \\
\nonumber
&- (n-k-1) \\
\hinv(P) &= | \{1 \leq i < j \leq n : \height_i(P) = \height_j(P), \lab_i(P) < \lab_j(P) \} | \\
\nonumber
&+ |\{1 \leq i < j \leq n : \height_i(P) = \height_j(P) + 1, \lab_i(P) > \lab_j(P) \} | .
\end{align}

\begin{figure}
\begin{tikzpicture}
\filldraw[yellow!40!white] (0,0) rectangle (0.5,0.5);
\filldraw[yellow!40!white] (0.5,0.5) rectangle (1,2.5);
\filldraw[yellow!40!white] (1,2.5) rectangle (1.5,3);
\draw[step=0.5cm, gray, very thin] (-0.001,0) grid (1.5,3);

\draw[blue, very thick] (0,0) -- (0,1.5) -- (0.5,1.5) -- (0.5,3) -- (1.5,3);

\node at (0.25, 0.25) {3};
\node at (0.25, 0.75) {4};
\node at (0.25, 1.25) {6};
\node at (0.75, 1.75) {1};
\node at (0.75, 2.25) {4};
\node at (0.75, 2.75) {5};

\node at (2.5,3.25) {$\areai_i$};
\node at (2.5,0.25) {0};
\node at (2.5,0.75) {1}; 
\node at (2.5,1.25) {1};
\node at (2.5,1.75) {0};
\node at (2.5,2.25) {0};
\node at (2.5, 2.75) {1};

\node at (3.5,3.275) {$\height_i$};
\node at (3.5,0.25) {0};
\node at (3.5,0.75) {1}; 
\node at (3.5,1.25) {2};
\node at (3.5,1.75) {2};
\node at (3.5,2.25) {3};
\node at (3.5, 2.75) {4};

\end{tikzpicture}
\caption{An example $P \in \ldyck^{\stack}_{6,2}$ with stack $S$ given by $\diag(S) = \{1, 2, 6\}$.  The boxes in the stack are shaded yellow. We have $\area(P) = 3$, $\winv(P) =1$, and $\hinv(P) = 0$. } 
\label{fig:stack}
\end{figure}

\begin{prop}
\label{prop:stacks}
We can construct bijections
\begin{align}
\phi_{n,k}&: \ldyck^{\fall}_{n,n-k-1} \to \ldyck^{\stack}_{n,k} \\
\psi_{n,k}&: \ldyck^{\val}_{n,n-k-1} \to \ldyck^{\stack}_{n,k}
\end{align}
such that
\begin{align}
\label{area-}
\area(\phi_{n,k}((P,F))) &= \area^{-}((P,F)) \\
\label{hinv}
\hinv(\phi_{n,k}((P,F))) &= \dinv(P) \\
\area(\psi_{n,k}((P,F))) &= \area(P) \\
\label{winv}
\winv(\psi_{n,k}((P,V))) &= \dinv^{-}((P,V)) .
\end{align}
and $x^P$ is preserved. As a result,
\begin{align}
\risepoly_{n,k}(x;q,t) &= \sum_{P \in \ldyck^{\stack}_{n,k}} q^{\hinv(P)} t^{\area(P)} x^P \\
\valpoly_{n,k}(x;q,t) &= \sum_{P \in \ldyck^{\stack}_{n,k}} q^{\winv(P)} t^{\area(P)} x^P .
\end{align}
\end{prop}

\begin{proof}
To define $\phi_{n,k}$, we take some $P \in \ldyck_{n}$, $F \subseteq \fall(P)$ with $|F| = n-k-1$. We begin with the leaning stack that consists entirely of diagonal steps between squares. Then, for each column $j \in F$, we remove the east step in column $j+1$ and move the square of $S$ in column $j+1$ one space to the left. The result is $\phi_{n,k}((P,F))$. To invert $\phi_{n,k}$, we simply ``push'' over all squares of the stack that appear directly above the square below them and insert east steps in the columns that were occupied by these squares. To see that $\phi_{n,k}$ cooperates with the statistics as proposed, we note that, for each $j \in F$, the process above removes $j$ squares from between $P$ and the diagonal. This proves \eqref{area-}. Equation \eqref{hinv} follows from the fact that $\height_i(\phi_{n,k}((P,F))) =  \areai_{i}(P)$ and the definitions given above.

Now we define $\psi_{n,k}$ for $P \in \ldyck_n$, $V \subseteq \val(P)$. We begin with the completely diagonal leaning stack again. For each $i \in V$, we remove the east step preceding the north step in row $i$ and move the square of $S$ in row $i$ one space to the left.  To invert $\psi_{n,k}$, we push over all vertical squares in the stack and insert east steps preceding the rows that were occupied by these squares. We notice that, for each row $i$, $\areai_i(P) = \areai_i(\psi_{n,k}((P,V)))$, so $\psi_{n,k}$ preserves area. Equation \eqref{winv} follows from the definitions of $\winv$ and $\dinv^{-}$. 

\end{proof}

Figure \ref{fig:bijections} contains examples of the maps $\phi_{n,k}$ and $\psi_{n,k}$. We note that the composition $\psi_{n,k} \, \circ \,  \phi_{n,k}$ is a bijection $\ldyck^{\fall}_{n,n-k-1} \to \ldyck^{\val}_{n,n-k-1}$ that preserves the monomial $x^P$. Furthermore, Proposition \ref{prop:stacks} implies $\risepoly_{n,k}(x;1,t) = \valpoly_{n,k}(x;1,t)$.

\subsection{Densely labeled Dyck paths}
\label{ssec:dense}

For our final combinatorial formulation, we again begin with integers $n > k \geq 0$. We use a shorter Dyck path $D \in \dyck_{k+1}$. Now we label each lattice square that occurs weakly above the line $y=x$ whose northwest corner intersects $D$. A square whose west edge is a north step of $D$ is called a north square; the other labeled squares are called east squares. Furthermore, we label these squares with sets of positive integers such that
\begin{enumerate}
\item no north square receives the label $\emptyset$,
\item for two north squares in the same column, every entry in the label of the lower square is less than every entry in the label of the upper square, and
\item there are $n$ total elements used in the labels.
\end{enumerate}
We call the resulting collection of objects \emph{densely labeled Dyck paths}, written \\ $\ldyck^{\dense}_{n,k}$. Figure \ref{fig:bijections} contains an example of a densely labeled Dyck path.

In order to move the statistics from our previous objects, for each element $r$ of any label in some $P \in \ldyck^{\dense}_{n,k}$ we set $\area(r, P)$ to be the number of full squares between $r$'s square and the diagonal. It is quite difficult to define the height of an entry in this setting, so we focus only on the $\area$ and $\winv$ statistics. We say
\begin{itemize}
\item $\area(P) = \sum_{\text{label entries } r} \area(r, P)$,
\item $\winv(P)$ is equal to the number of pairs of label entries $(r,s)$ with $r$ minimal in its square, $r$'s square appearing strictly west of $s$'s square, and either
\begin{itemize}
\item $r < s$ and $\area(r, P) = \area(s, P)$, or
\item $r > s$ and $\area(r, P) = \area(s, P) + 1$
\end{itemize}
minus the number of entries in labels in east squares in $P$.
\end{itemize}

\begin{prop}
\label{prop:dense}
We can construct a bijection $\theta_{n,k} : \ldyck^{\stack}_{n,k} \to \ldyck^{\dense}_{n,k}$ such that
\begin{align}
\area(\theta_{n,k}(P)) &= \area(P) \\ 
\winv(\theta_{n,k}(P)) &= \winv(P).
\end{align}
As a result, 
\begin{align}
\valpoly_{n,k}(x;q,t) &= \sum_{P \in \ldyck^{\dense}_{n,k}} q^{\winv(P)} t^{\area(P)} x^P .
\end{align}
\end{prop}

\begin{proof}
We define $\theta_{n,k}$ by contracting every north step of $P$ that shares a row with a vertical square of the leaning stack. The labels whose north steps are removed are simply combined with the remaining labels to form the set labels. The inverse is direct and the assertions about the statistic follow from the definitions above. 
\end{proof}

We summarize all of our bijections in Figure \ref{fig:bijections}.

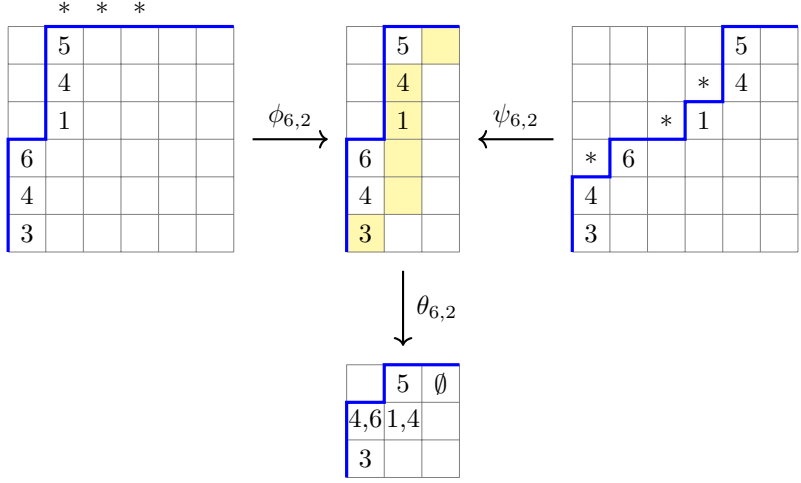
\begin{figure}
\begin{tikzpicture}
\filldraw[yellow!40!white] (0,0) rectangle (0.5,0.5);
\filldraw[yellow!40!white] (0.5,0.5) rectangle (1,2.5);
\filldraw[yellow!40!white] (1,2.5) rectangle (1.5,3);
\draw[step=0.5cm, gray, very thin] (-0.001,0) grid (1.5,3);
\draw[blue, very thick] (0,0) -- (0,1.5) -- (0.5,1.5) -- (0.5,3) -- (1.5,3);
\node at (0.25, 0.25) {3};
\node at (0.25, 0.75) {4};
\node at (0.25, 1.25) {6};
\node at (0.75, 1.75) {1};
\node at (0.75, 2.25) {4};
\node at (0.75, 2.75) {5};

\draw[thick, ->] (-1.25, 1.5) -- (-0.25, 1.5);
\node at (-0.75, 1.8) {$\phi_{6,2}$};

\draw[step=0.5cm, gray, very thin] (-4.501, 0) grid (-1.5, 3);
\draw[blue, very thick] (-4.5, 0) -- (-4.5, 1.5) -- (-4, 1.5) -- (-4, 3) -- (-1.5, 3);
\node at (-4.25, 0.25) {3};
\node at (-4.25, 0.75) {4};
\node at (-4.25, 1.25) {6};
\node at (-3.75, 1.75) {1};
\node at (-3.75, 2.25) {4};
\node at (-3.75, 2.75) {5};
\node at (-3.75, 3.25) {$\ast$};
\node at (-3.25, 3.25) {$\ast$};
\node at (-2.75, 3.25) {$\ast$};

\draw[thick, ->] (2.75, 1.5) -- (1.75, 1.5);
\node at (2.25, 1.8) {$\psi_{6,2}$};

\draw[step=0.5cm, gray, very thin] (2.999, 0) grid (6, 3);
\draw[blue, very thick] (3,0) -- (3,1) -- (3.5, 1) -- (3.5, 1.5) -- (4.5, 1.5) -- (4.5, 2) -- (5, 2) -- (5, 3) -- (6, 3);
\node at (3.25, 0.25) {3};
\node at (3.25, 0.75) {4};
\node at (3.75, 1.25) {6};
\node at (4.75, 1.75) {1};
\node at (5.25, 2.25) {4};
\node at (5.25, 2.75) {5};
\node at (3.25, 1.25) {$\ast$};
\node at (4.25, 1.75) {$\ast$};
\node at (4.75, 2.25) {$\ast$};

\draw[thick, ->] (0.75, -0.25) -- (0.75, -1.25);
\node at (1.2, -0.75) {$\theta_{6,2}$};

\draw[step=0.5cm, gray, very thin] (0,-3) grid (1.5, -1.5);
\draw[blue, very thick] (0, -3) -- (0, -2) -- (0.5, -2) -- (0.5, -1.5) -- (1.5, -1.5);
\node at (0.25, -2.75) {3};
\node at (0.25, -2.25) {4,6};
\node at (0.75, -2.25) {1,4};
\node at (0.75, -1.75) {5};
\node at (1.25, -1.75) {$\emptyset$};

\end{tikzpicture}
\caption{Examples of the maps $\phi_{6,2}$, $\psi_{6,2}$, and $\theta_{6,2}$. We have marked the selected double falls and contractible valleys with stars.}
\label{fig:bijections}
\end{figure}

\subsection{The $q=t=1$ case}
\label{ssec:q=t=1}
As an immediate application of these interpretations, we obtain a formula for $\risepoly_{n,k}(x;1,1) + \risepoly_{n,k-1}(x;1,1)$. Thanks to the leaning stacks interpretation, we already know that $\risepoly_{n,k}(x;1,t) = \valpoly_{n,k}(x;1,t)$. In Section \ref{sec:t=1/q}, we will see that $\Delta_{e_{k}} e_n$ also obeys the formula we prove here. This proves the $q=t=1$ case of the Delta Conjecture.

\begin{prop}
\label{prop:q=t=1}
For any integers $1 \leq k \leq n$, 
\begin{align}
\risepoly_{n,k}(x;1,1) + \risepoly_{n,k}(x;1,1) &= \left.  \frac{1}{k+1} \binom{n}{k} \left(\sum_{i \geq 0} e_i u^i \right)^{k+1} \right|_{u^n} \\
&= \frac{1}{k+1} \binom{n}{k} e_{n}[(k+1)X] . 
\end{align}

\end{prop}

\begin{proof}
By the definition of $\risepoly_{n,k}(x;q,t)$ appearing in the Delta Conjecture,
\begin{align}
\label{q=t=1-start}
\risepoly_{n,k}(x;1,1) + \risepoly_{n,k-1}(x;1,1) &= \left. \sum_{P \in \ldyck_n} (1+z)^{|\rise(P)| + 1} x^P \right|_{z^{n-k}} .
\end{align}
Given a partition $\lambda \vdash n$, set $c_i = c_i(\lambda)$ to be the multiplicity of $i$ in $\lambda$. As mentioned in Equation 4 of \cite{rational-catalan}, the number of Dyck paths with exactly $c_i$ vertical runs of length $i$ for each $i$ is 
\begin{align}
\frac{1}{n+1} \binom{n+1}{c_1, c_2, \ldots, c_n, n-\ell(\lambda)+1} .
\end{align}
Furthermore, such a Dyck path has $n-\ell(\lambda)$ double rises. We label each of the vertical runs of such a Dyck path with increasing sequences of integers, contributing an $e_{\lambda}$ term. Hence \eqref{q=t=1-start}
\begin{align}
&= \left. \sum_{\lambda \vdash n} \frac{1}{n+1} \binom{n+1}{c_1, c_2, \ldots, c_n, n-\ell(\lambda)+1} (1+z)^{n-\ell(\lambda)+1} e_{\lambda} \right|_{z^{n-k}} \\
&= \left. \frac{1}{n+1} \binom{n+1}{n-k} \left(\sum_{i \geq 0} e_i u^i \right)^{k+1} \right|_{u^n} 
\end{align}
which proves the first identity in the proposition. The second identity is a consequence of Cauchy's Formula.
\end{proof}

\section{Ordered set partitions and the $q=0$, $t=0$, and $q=1$ cases}
\label{sec:osp}
In this section, we show how previous work of the second and third authors in \cite{rw, tesler-wilson} can be combined to prove the following result.

\begin{thm}
\label{thm:zero}
 The coefficients of the monomial quasisymmetric function $M_{1^n}$ are equal in each of the following:
\begin{align}
\risepoly_{n,k}(x;q,0), \ \ 
\risepoly_{n,k}(x;0,q), \ \ 
\valpoly_{n,k}(x;q,0), \ \
\left. \Delta^{\prime}_{e_k} e_n \right|_{t=0}, \ \ 
\left. \Delta^{\prime}_{e_k} e_{n} \right|_{q=0, \, t=q}.
\end{align}
\end{thm}

It is notable that $\valpoly_{n,k}(x;0,q)$ is not included in the list above; we will explain why this case has proved more difficult than the others at the end of this section. We note that \cite{tesler-wilson} also contains a proof that
\begin{align}
\left. \left\langle \Delta^{\prime}_{e_k} e_n, p_{1^n} \right\rangle \right|_{q=1} &= \left\langle \risepoly_{n,k}(x;1,t), p_{1^n} \right\rangle .
\end{align}

The \emph{ordered set partitions} of order $n$ with $k$ blocks are partitions of the set $\{1,2,\ldots,n\}$ into $k$ subsets (called blocks) with some order on the blocks. We write this set as $\osp{n}{k}$. More generally, given a composition $\alpha$ of length $n$, the ordered multiset partitions $\osp{\alpha}{k}$ are the partitions of the multiset $\{i^{\alpha_i}: 1 \leq i \leq n\}$ into $k$ ordered blocks. In \cite{tesler-wilson}, the third author showed that 
\begin{align}
 \left. \Delta^{\prime}_{e_k} e_n \right|_{M_{1^n}, \, t=0} &=   \left. \Delta^{\prime}_{e_k} e_{n} \right|_{M_{1^n}, \, q=0, \, t=q} = \sum_{\pi \in \osp{n}{k+1}} q^{\inv(\pi)} 
\end{align}
where $\inv(\pi)$ counts the number of pairs $a > b$ such that $a$'s block is strictly to the left of $b$'s block in $\pi$ and $b$ is minimal in its block in $\pi$. For example, $15|23|4$ has two inversions, between the 5 and the 2 and the 5 and the 4. We claim that setting one of $q$ or $t$ equal to zero in our combinatorial interpretations also yields a sum involving ordered partitions. To make this more precise, we define three more statistics on ordered multiset partitions. 

First, given some $\pi \in \osp{\alpha}{k}$ we number $\pi$'s blocks $\pi_1, \pi_2, \ldots, \pi_k$ from left to right. Let $\pi^h_i$ be the $h$th smallest element in $\pi_i$, beginning at $h=0$. Then the \emph{diagonal inversions} of $\pi$, written $\Dinv(\pi)$, are the triples
\begin{align}
\{(h, i, j) : 1 \leq i < j \leq k, \ \pi^h_i > \pi^h_j \} \cup \{(h, i, j): 1 \leq i < j \leq k, \ \pi^h_i < \pi^{h+1}_j \} .
\end{align}
The triples of the first type are \emph{primary diagonal inversions}, and the triples of the second type are \emph{secondary diagonal inversions}. We set $\dinv(\pi)$ to be the cardinality of $\Dinv(\pi)$. 

To define the \emph{major index} of $\pi$, we consider the permutation $\sg = \sg(\pi)$ obtained by writing each block of $\pi$ in decreasing order. Then we recursively form a word $w$ by setting $w_0=0$ and $w_i = w_{i-1} + \chi(\sg_i \text{ is minimal in its block in } \pi)$. Then we set
\begin{align}
\maj(\pi) &= \sum_{i : \ \sg_i > \sg_{i+1} } w_i .
\end{align}

Finally, we define the \emph{minimum major index} of $\pi$ as follows. We begin by writing the elements of $\pi_k$ in increasing order from left to right. Then, recursively for $i = k-1$ to $1$, we choose $r$ to be the largest element in $\pi_i$ that is less than or equal to the leftmost element in $\pi_{i+1}$, as previously recorded. If there is no such $r$, we write $\pi_i$ in increasing order. If there is such an $r$, beginning with $\pi_i$ in increasing order, we cycle its elements until $r$ is the rightmost element in $\pi_i$. We write down $\pi_i$ in this order. We continue this process until we have processed each block of $\pi$. For example, consider the ordered multiset permutation $\pi = 13|23|14|234$. Processing the blocks of $\pi$ from right to left, we obtain $312341234$.  We consider the result as a permutation, which we denote $\tau = \tau(\pi)$, and define
\begin{align}
\minimaj(\pi) = \sum_{i: \ \tau_i  > \tau_{i+1}} i
\end{align}
i.e.\ the major index of the permutation $\tau$. The name $\minimaj$ comes from the fact that $\minimaj(\pi)$ is equal to the minimum major index achieved by any permutation that can be obtained by permuting elements within the blocks of $\pi$. 

\begin{prop}
\label{prop:omp}
\begin{align}
\label{omp-dinv}
\left. \risepoly_{n,k}(x;q,0) \right|_{M_{\alpha}}  &= \sum_{\pi \in \osp{\alpha}{k+1}} q^{\dinv(\pi)} \\ 
\label{omp-maj}
\left. \risepoly_{n,k}(x;0,q) \right|_{M_{\alpha}}  &= \sum_{\pi \in \osp{\alpha}{k+1}} q^{\maj(\pi)} \\ 
\label{omp-inv}
\left. \valpoly_{n,k}(x;q,0) \right|_{M_{\alpha}}  &= \sum_{\pi \in \osp{\alpha}{k+1}} q^{\inv(\pi)} \\
\label{omp-minimaj}
\left. \valpoly_{n,k}(x;0,q) \right|_{M_{\alpha}}  &= \sum_{\pi \in \osp{\alpha}{k+1}} q^{\minimaj(\pi)} .
\end{align}
Since taking the coefficient of $M_{1^n}$ is equivalent to taking the inner product with $p_{1^n}$ for a symmetric function in $\Lambda^{(n)}$, this completes the proof of Theorem \ref{thm:zero}.
\end{prop}

\begin{proof}
To prove \eqref{omp-dinv}, it is easiest to use the interpretation of $\risepoly_{n,k}(x;q,0)$ involving leaning stacks given in Subsection \ref{ssec:stacks}, which gives
\begin{align}
\left. \risepoly_{n,k}(x;q,0) \right|_{M_{\alpha}} &= \sum_{P} q^{\hinv(P)}
\end{align}
where the sum is over $P \in \ldyck^{\stack}_{n,k}$ with $\area(P) = 0$ and $x^P = \prod_{i=1}^{\ell(\alpha)} x_i^{\alpha_i}$. We consider the map from such paths $P$ to ordered multiset partitions $\pi \in \osp{\alpha}{k+1}$ where $\pi_i$ consists of the elements in the $i$th column of $P$, counting from right to left. This is clearly a bijection, and it follows from the definitions that $\hinv(P) = \dinv(\pi)$, proving \eqref{omp-dinv}. 

To prove \eqref{omp-maj}, we consider the interpretation of $\risepoly_{n,k}(x;q,t)$ from Subsection \ref{ssec:decorating} in which we decorated double rises. This allows us to write 
\begin{align}
\left. \risepoly_{n,k}(x;0,q) \right|_{M_{\alpha}} &= \sum_{P} q^{\area(P)}
\end{align}
where the sum is over $(P,R) \in \ldyck^{\rise}_{n,n-k-1}$ with $\dinv(P) = 0$ and  $x^P = \prod_{i=1}^{\ell(\alpha)} x_i^{\alpha_i}$. 
We note that $P$ can only have $\dinv(P) =0 $ if $\areai_i(P)$ is weakly increasing from bottom to top; furthermore, $\areai_{i+1}(P) > \areai_i(P)$ if and only if $\lab_{i+1}(P) > \lab_i(P)$. To form an ordered multiset partition from such a path $P$, we record the labels of $P$ from top to bottom as a multiset permutation $\sg$. Then, for each $i \in R$, we join the corresponding entry of $\sg$ with the entry to its right to form a block. This map gives a bijection to $\osp{\alpha}{k+1}$ and sends $\area$ to $\maj$.

For \eqref{omp-inv}, we consider the interpretation of $\valpoly_{n,k}(x;q,t)$ from Subsection \ref{ssec:dense} involving densely labeled Dyck paths, which implies
\begin{align}
\left. \valpoly_{n,k}(x;q,0) \right|_{M_{\alpha}} &= \sum_{P} q^{\winv(P)}
\end{align}
where the sum is over $P \in \ldyck^{\dense}_{n,k}$ with $\area(P) = 0$ and $x^P =  \prod_{i=1}^{\ell(\alpha)} x_i^{\alpha_i}$. $\area(P) = 0$ implies that the underlying Dyck path of $P$ is the path $(NE)^{k+1}$ that never leaves the diagonal. To form $\pi \in \osp{\alpha}{k}$, we simply make each label set of $P$ from right to left into a block. This is a bijection and it is clear that $\winv(P) = \inv(\pi)$. 

The proof of \eqref{omp-minimaj} is quite technical, so we have placed it in Appendix \ref{app:minimaj}.

\end{proof}

The reason that $\valpoly_{n,k}(x;0,q)$ does not appear in Theorem \ref{thm:zero} is that the $\minimaj$ statistic behaves quite differently. In an earlier preprint, we conjectured the following, which has since been proved by Brendon Rhoades \cite{rhoades-minimaj}.

\begin{prop}
\label{prop:minimaj}
\begin{align}
\sum_{\pi \in \osp{\alpha}{k+1}} q^{\minimaj(\pi)} &= \sum_{\pi \in \osp{\alpha}{k+1}} q^{\dinv(\pi)} 
=  \sum_{\pi \in \osp{\alpha}{k+1}} q^{\maj(\pi)} 
=  \sum_{\pi \in \osp{\alpha}{k+1}} q^{\inv(\pi)} .
\end{align}
\end{prop}

\section{Results at $t=1/q$}
\label{sec:t=1/q}

In this section, we consider the special case $t=1/q$. As in the Shuffle Conjecture, this case is much more approachable from the symmetric function point of view than the general setting. In particular, it is not difficult to obtain a plethystic formula for $\Delta_{e_k} e_n$ at $t=1/q$. We prove the plethystic formula below and then use it to show that $\Delta_{e_k} e_n$ is Schur positive at $t=1/q$ up to a power of $q$.

\begin{thm}
\label{thm:t=1/q}
For any symmetric function $f \in \Lambda^{(k)}$, 
\begin{align}
\left. \Delta_{f} e_{n} \right|_{t=1/q} =& \frac{f[\qint{n}] e_{n} [ X \qint{k+1} ]}{q^{k(n-1)} \qint{k+1} } .
\end{align}
\end{thm}

\begin{proof}
First we note that
\begin{align}
\tilde{H}_{\mu}[X; q, 1/q] =& \   C s_{\mu}\left[\frac{X}{1-q}\right]
\end{align}
for a constant $C$. This fact can be derived from \cite{macdonald}. We use Cauchy's Formula to write
\begin{align}
e_{n}[X] =& \   e_{n} \left[(1-q)\frac{X}{1-q}\right] = \sum_{\mu \vdash n} s_{\mu^{\prime}}\left[\frac{X}{1-q}\right] s_{\mu}[1-q].
\end{align}
For any monomial $u$, $s_{\mu}[1-u]$ is zero if $\mu$ is not a hook shape and 
\begin{align}
s_{\mu}[1-u] =& \   (-u)^{r}(1-u)
\end{align}
if $\mu = (n-r, 1^{r})$ \cite{macdonald}. Therefore, summing over hook shapes $\mu$, we have
\begin{align}
\label{en}
e_{n}[X] =& \   \sum_{\mu = (n-r, 1^{r})} s_{\mu^{\prime}} \left[\frac{X}{1-q}\right] (-q)^{r}(1-q) .
\end{align}
Next, we note that, for $\mu = (n-r, 1^r)$, $\mu^{\prime} = (r+1, 1^{n-r-1})$ and
\begin{align}
B_{\mu^{\prime}}(q,1/q) =& \   q^{-(n-r-1)} \qint{n} .
\end{align}
Therefore
\begin{align}
\label{eigenvalue}
f[B_{\mu^{\prime}}(q, 1/q)] =& \   q^{-k(n-r-1)}f[\qint{n}].
\end{align}
Combining \eqref{en} with \eqref{eigenvalue}, we see that $\left. \Delta_{f} e_{n}[X] \right|_{t=1/q}$ is equal to
\begin{align}
&\sum_{\mu = (n-r, 1^{r})} q^{-k(n-r-1)}(-q)^{r}(1-q) f[\qint{n}] s_{\mu^{\prime}} \left[\frac{X}{1-q}\right] \\
=& \   \frac{f[\qint{n}]}{q^{k(n-1)} \qint{k+1}}  \sum_{\mu = (n-r, 1^{r})} (-q^{k+1})^{r}(1-q^{k+1}) s_{\mu^{\prime}} \left[ \frac{X}{1-q} \right] .
\end{align}
Applying Cauchy's Formula again, we get 
\begin{align}
\sum_{\mu = (n-r, 1^{r})} (-q^{k+1})^{r}(1-q^{k+1}) s_{\mu^{\prime}} \left[\frac{X}{1-q}\right] =& \   e_{n} \left[ X \qint{k+1} \right] . 
\end{align}
\end{proof}

From Theorem \ref{thm:t=1/q}, it is easy to compute
\begin{align}
\left. \Delta_{e_k} e_{n} \right|_{t=1/q} &= \frac{q^{\binom{k}{2} - k(n-1)}}{\qint{k+1}} \qbinom{n}{k} e_{n}[X \qint{k+1}].
\end{align}
In work in preparation, the third author uses this formula along with combinatorial recursions to prove both versions of the Shuffle Conjecture after setting $t=1/q$ and taking the scalar product with $e_{n-r} h_r$ or $h_{n-r} h_r$ for any nonnegative integer $r$. We can also use Theorem \ref{thm:t=1/q} along with a recent result of Garsia, Leven, Wallach, and Xin to prove a Schur positivity result for our symmetric function at $t=1/q$.

\begin{cor}
\label{cor:schur_positivity}
$\left. q^{k(n-1) - \binom{k}{2}}\Delta_{e_k} e_n \right|_{t=1/q}$ is a Schur positive symmetric polynomial.
\end{cor}

\begin{proof}
Let $d = \gcd(k+1,n)$. Then Theorem 2.1 in \cite{rational-t=1/q} implies that
\begin{align}
\frac{\qint{d}}{\qint{k+1}} e_{n} \left[ X \qint{k+1} \right]
\end{align}
is a Schur positive symmetric polynomial. By Theorem \ref{thm:t=1/q}, it is enough to show that $\frac{1}{\qint{d}} \qbinom{n}{k} \in \mathbb{N}[q]$. Furthermore, Proposition 2.4 in \cite{rational-t=1/q} implies that if $\frac{1}{\qint{d}} \qbinom{n}{k}$ is a polynomial then it must have nonnegative coefficients (since $\qbinom{n}{k}$ is known to be a unimodal positive polynomial). Therefore we only need to show that $\frac{1}{\qint{d}} \qbinom{n}{k}$ is a polynomial. 

To accomplish this, we will use the $q$-Lucas Theorem, apparently first proved in \cite{q-lucas} and given a nice combinatorial proof in \cite{sagan-congruence}. To state the $q$-Lucas Theorem, given integers $n$, $k$, and $p$, we divide $n$ and $k$ by $p$ to obtain $n = n_1 p + n_0$ and $k = k_1 p + k_0$. Then 
\begin{align}
\qbinom{n}{k} \equiv \binom{n_1}{k_1} \qbinom{n_0}{k_0} \pmod{ \Phi_p(q) } .
\end{align}
where $\Phi_p(q)$ is $p$th cyclotomic polynomial. Consider any $p$ such that $\Phi_p(q)$ divides $\qint{d}$. If we can show that all such $\Phi_p(q)$ divide $\qbinom{n}{k}$, we are done. Since $p$ divides $d$ and $d$ divides both $n$ and $k+1$, $p$ divides $n$ but it does not divide $k$. This means that $n_0 = 0$ and $k_0 > 0$. By the $q$-Lucas Theorem,
\begin{align}
\qbinom{n}{k} &\equiv \binom{n_1}{k_1} \qbinom{n_0}{k_0} \equiv 0 \pmod{ \Phi_p(q)} 
\end{align} 
so $\Phi_p(q)$ divides $\qbinom{n}{k}$. 
\end{proof}

\section{Proof of the Rise Version at $k=1$}
\label{sec:k=1}
In this section, we prove the following special case of the Rise Version of the Delta Conjecture.

\begin{thm}
\label{thm:k=1}
For any positive integer $n$,
\begin{align}
\Delta_{e_{1}} e_n &= \risepoly_{n,0}(x;q,t) + \risepoly_{n,1}(x;q,t) \\
&= \sum_{m=0}^{\lfloor n/2 \rfloor} s_{2^m, 1^{n-2m}} \sum_{p=m}^{n-m} \qtint{p} .
\end{align}
This verifies \eqref{rise-intro} from Conjecture \ref{conj:delta} for $k=1$. 
\end{thm}

We deal with the symmetric function component of Theorem \ref{thm:k=1} in Subsection \ref{ssec:k=1-sym} and the combinatorial component in Subsection \ref{ssec:k=1-comb}. 

\subsection{The symmetric side}
\label{ssec:k=1-sym}

In this subsection, we prove the ``symmetric side'' of Theorem \ref{thm:k=1}, restated below.

\begin{prop}
\label{prop:k=1-sym}
For any positive integer $n$,
\begin{align}
\Delta_{e_1} e_n &= \sum_{m=0}^{\lfloor n/2 \rfloor} s_{2^m, 1^{n-2m}} \sum_{p=m}^{n-m} \qtint{p} .
\end{align}
\end{prop}

Our main tool will be the following reciprocity rule for the operator $\Delta$, which was proved by the first author as Corollary 2 in \cite{schroder}. 

\begin{lemma}[Corollary 2 in \cite{schroder}]
\label{lemma:reciprocity}
For positive integers $d, n$ and any symmetric function $f \in \Lambda^{(n)}$,
\begin{align}
\left\langle \Delta_{e_{d-1}} e_n, f \right\rangle &= \left\langle \Delta_{\omega f} e_d, s_d \right \rangle.
\end{align}
\end{lemma}

We set $d=2$ and $f = s_{\lambda}$ for $\lambda \vdash n$, since taking the scalar product of a symmetric function with $s_{\lambda}$ yields the coefficient of $s_{\lambda}$ in the Schur expansion of that symmetric function. Lemma \ref{lemma:reciprocity} implies that 
\begin{align}
\left\langle \Delta_{e_{1}} e_n, s_{\lambda} \right\rangle &= \left\langle \Delta_{s_{\lambda^{\prime}} } e_2, s_2 \right \rangle.
\end{align}
We can compute the right-hand side by hand. First, we expand $e_2$ into the modified Macdonald polynomial basis:
\begin{align}
e_2 &= \frac{1}{t-q} \tilde{H}_{1,1} - \frac{1}{t-q} \tilde{H}_{2}.
\end{align}
Then we apply the operator $\Delta_{s_{\lambda^{\prime}}}$.
\begin{align}
\Delta_{s_{\lambda}} e_2 &= \frac{s_{\lambda^{\prime}}[1+t]}{t-q} \tilde{H}_{1,1} - \frac{s_{\lambda^{\prime}}[1+q]}{t-q} \tilde{H}_{2} .
\end{align}
Now we expand this expression into the Schur basis and take the coefficient of $s_{2}$, yielding
\begin{align}
\left\langle \Delta_{s_{\lambda^{\prime}}} e_2, s_2 \right \rangle &= \frac{s_{\lambda^{\prime}}[1+t] - s_{\lambda^{\prime}}[1+q]}{t-q} .
\end{align}
It is already clear that the above expression is a polynomial in $q$ and $t$. Moreover, for any monomial $u$ the principal specialization $s_{\lambda^{\prime}}[1+u]$ is equal to the sum $\sum_{T} u^{\# \text { 2's in $T$}}$ over all semi-standard tableaux $T$ of shape $\lambda^{\prime}$ filled with $1$'s and $2$'s. This sum is zero if $\lambda^{\prime}$ has more than two rows, so we can restrict our attention to $\lambda^{\prime} = (n-m, m)$ for some integer $0 \leq m \leq \lfloor n/2 \rfloor$. For such a tableaux $T$ of shape $(n-m,m)$, it is clear that the first $m$ entries in the first row of $T$ must be 1's and all entries in the second row of $T$ must be 2's. Of the remaining $n-2m$ entries, we are free to choose an integer $0 \leq i \leq n-2m$ such that the left $i$ entries are 1's and the right $n-2m-i$ entries are 2's. Hence
\begin{align}
s_{n-m,m}[1+u] &= \sum_{p=m}^{n-m} u^p 
\end{align}
Since $(n-m,m)^{\prime} = (2^m, 1^{n-2m})$, we have
\begin{align}
\left\langle \Delta_{e_1} e_n, s_{2^m, 1^{n-2m}} \right\rangle &= \frac{\sum_{p=m}^{n-m} t^p - q^p}{t-q} = \sum_{p=m}^{n-m} \qtint{p}
\end{align}
which proves Proposition \ref{prop:k=1-sym}. 

In theory, this method can be used to compute $\Delta_{e_k} e_n$ for any fixed value of $k$. For example, $\left\langle \Delta_{e_2} e_n, s_{\lambda} \right\rangle$ equals
\begin{align}
\frac{(t-q^2)s_{\lambda^{\prime}}[1+t+t^2] - (q+t+1)(t-q) s_{\lambda^{\prime}}[1+q+t] + (t^2-q) s_{\lambda^{\prime}}[1+q+q^2]}{(t-q)(t^2-q)(t-q^2)} 
\end{align}
which is clearly a polynomial in $q$ and $t$. Unfortunately, it is not clear why the resulting expression should be a \emph{positive} polynomial in $q$ and $t$; furthermore, this problem only gets more difficult as $k$ grows. 

\subsection{The combinatorial side}
\label{ssec:k=1-comb}

In this subsection, we prove the following proposition, completing the proof of Theorem \ref{thm:k=1}.

\begin{prop}
\label{prop:k=1-comb}
For any positive integer $n$,
\begin{align}
\risepoly_{n,0}(x;q,t) + \risepoly_{n,1}(x;q,t) &= \sum_{m=0}^{\lfloor n/2 \rfloor} s_{2^m, 1^{n-2m}} \sum_{p=m}^{n-m} \qtint{p} .
\end{align}
\end{prop}

First, we note that $\risepoly_{n,0}(x;q,t) = s_{1^n}$, which accounts for the $m=p=0$ term above. We will need to work harder to expand $\risepoly_{n,1}(x;q,t)$. Recall the interpretation for $\risepoly_{n,k}(x;q,t)$ given in terms of labeled Dyck paths and leaning stacks in Subsection \ref{ssec:stacks}:
\begin{align}
\risepoly_{n,k}(x;q,t) &= \sum_{P \in \ldyck^{\stack}_{n,k}} q^{\hinv(P)} t^{\area(P)} x^P .
\end{align}
We note that this interpretation is closely related to the LLT polynomials of \cite{llt}. Namely, we can refine the sum on the right-hand side by fixing a leaning stack $S$ and then a Dyck path $D \in \dyck(S)$ and considering all ways of labeling the Dyck path $D$. Thus
\begin{align}
\risepoly_{n,k}(x;q,t) &= \sum_{S \in \stack_{n,k}} \sum_{D \in \dyck(S)} t^{\area(D)} \sum_{P \in \ldyck(S) :  \,D(P) = D} q^{\hinv(P)} x^P \\
\label{hag-symmetry}
&= \sum_{S \in \stack_{n,k}} \sum_{D \in \dyck(S)} t^{\area(D)} \llt_{S,D}(x;q) 
\end{align}
where we have defined
\begin{align}
\llt_{S,D}(x;q) &=  \sum_{P \in \ldyck(S) :  \, D(P) = D} q^{\hinv(P)} x^P.
\end{align}
We call this the \emph{LLT polynomial} with respect to $S$ and $D$, since these are special cases of the polynomials introduced in \cite{llt}. We can relate our versions of LLT polynomials more precisely to the notation for LLT polynomials appearing in \cite{hhl} as follows. Say that the north steps of $D$ appear in $d$ different columns and that the bottom row $i$ in the $j$th column (from right to left, beginning with $j=1$) has $\height_i = c_j$. Consider the tuple of skew diagrams $\nu = (\nu^{(1)}, \ldots, \nu^{(d)})$ where $\nu^{(j)}$ has number of squares equal to the number of north steps in the $j$th column of $D$ with the content of the bottom square equal to $c_j$. Then $\llt_{S,D}(x;q) = G_{\nu}(x;q)$, where the latter appears as Definition 3.2 of \cite{hhl}.

There are many benefits of this connection between $\risepoly_{n,k}(x;q,t)$ and LLT polynomials. The first is that LLT polynomials are known to be symmetric; this fact, along with \eqref{hag-symmetry}, implies that $\risepoly_{n,k}(x;q,t)$ is symmetric. On the other hand, we are still unable to prove that $\valpoly_{n,k}(x;q,t)$ is symmetric. More pertinent to our current case, when $D$ has two columns, much is known about the LLT polynomial $\llt_{S,D}(x;q)$. In the remainder of this subsection, we leverage this information to complete the proof of Theorem \ref{thm:k=1}. 

We use the notation that the \emph{reading word} of a labeled Dyck path $P \in \ldyck^{\stack}_{n,k}$, written $\word(P)$, is obtained by reading its labels from maximum $\height_i$ value down to $\height_i = 0$ from right to left. For example, the reading word of the $P \in \ldyck^{\stack}_{6,2}$ pictured in Figure \ref{fig:diagram} is $541643$.  We say that a word whose entries are positive integers is \emph{Yamanouchi} if each of its suffixes has more $i+1$'s than $i$'s for every positive integer $i$. 

\begin{lemma}[Carr\'e and Leclerc \cite{carre-leclerc}, van Leeuwen \cite{van-leeuwen}]
\label{lemma:yam}
For any $S \in \stack_{n,1}$ and $D \in \ldyck(S)$, the coefficient of $s_{\lambda}$ in the Schur expansion of  $\llt_{S,D}(x;q)$ is equal to the sum
\begin{align}
\sum_{P} q^{\hinv(P)}
\end{align}
over all $P \in \ldyck(S)$ with $D(P) = D$ such that $x^P = \prod_{i=1}^{\ell(\lambda)} x_i^{\lambda_i}$ and $\word(P)$ is Yamanouchi.
\end{lemma}

For any such $P$, each integer can be used as a label at most twice. Thus the only Schur functions appearing in the expansion of $\llt_{S,D}(x;q)$ are of the form $s_{2^m, 1^{n-2m}}$ for some integer $0 \leq m \leq \lfloor n/2 \rfloor$. Furthermore, we can uniquely represent a labeled Dyck path $P$ that satisfies the conditions in Lemma \ref{lemma:yam} by filling a certain two-column array with $X$'s and $Y$'s according to the following procedure. For each height that occurs in $P$ from 0 up to the maximum height, consider the two columns of $P$. If the left column of $P$ contains a label at that height, place a square into the left column of the array. If we have already come across the value of the label while creating our array, we place a $Y$ in the new square; otherwise, we place an $X$. Then we do the same for the right column. We continue until all heights have been processed. We call this the \emph{XY diagram} of $P$.  

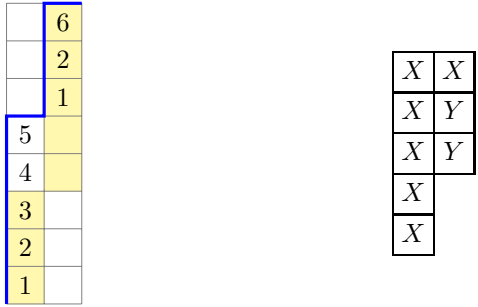
\begin{figure}
\begin{minipage}{0.4\linewidth}
\centering
    \begin{tikzpicture}
\filldraw[yellow!40!white] (0,0) rectangle (0.5,1.5);
\filldraw[yellow!40!white] (0.5,1.5) rectangle (1,4);
\draw[step=0.5cm, gray, very thin] (-0.001,0) grid (1,4);
\draw[blue, very thick] (0,0) -- (0,2.5) -- (0.5,2.5) -- (0.5,4) -- (1,4);
\node at (0.25,0.25) {1};
\node at (0.25,0.75) {2};
\node at (0.25,1.25) {3};
\node at (0.25,1.75) {4};
\node at (0.25,2.25) {5};
\node at (0.75,2.75) {1};
\node at (0.75,3.25) {2};
\node at (0.75,3.75) {6};
\end{tikzpicture}
\end{minipage}
\begin{minipage}{0.4\linewidth}
\centering
\begin{ytableau}
X & X \\
X & Y \\
X & Y \\
X & \none \\
X & \none 
\end{ytableau}
\end{minipage}

\caption{To the left we have drawn a two-column labeled Dyck path whose word is Yamanouchi with its leaning stack shaded yellow. To the right we have drawn the corresponding $XY$ diagram. }
\label{fig:diagram}
\end{figure}

Since each label appears at most twice in $P$, this process is well-defined. Furthermore, it is invertible; to obtain the original labeled Dyck path $P$, we scan the $XY$ diagram from bottom to top and left to right. For each $X$ or $Y$, we place a label in the corresponding column at the corresponding height that counts the number of times (including the current letter) that we have observed the current letter so far. 

It is clear that all $XY$ diagrams have two columns, that the left column may extend below the right column (but not vice versa), and that the lower left square of a diagram always contains an $X$. The crux of the proof of Proposition \ref{prop:k=1-comb} is that we can use the Yamanouchi restriction on $\word(P)$ to completely classify the possible $XY$ diagrams. We note that $\word(P)$ is Yamanouchi if and only if, reading the diagram from bottom to top and left to right, we have always seen at least as many $X$'s as $Y$'s. Furthermore, the labels of $P$ are increasing up columns if and only if there are no $Y$'s on top of $X$'s. Together with the Yamanouchi condition, this implies that $Y$'s always occur in the right column.

These conditions are enough to allow us to classify the possible $XY$ diagrams. From bottom to top, every diagram begins with $a \geq 0$ rows consisting of only a left square which contains an $X$. Then it has $b \geq 0$ rows which have two squares where the left square contains an $X$ and the right square contains a $Y$. From this point on, the diagram can have one of two types. Type I $XY$ diagrams have a sequence of $c \geq 0$ rows with two squares, both of which contain $X$'s, followed by a sequence of $d \geq 0$ rows with a single square containing an $X$. The final $d$ rows must either consist entirely of left squares or of right squares. In Type II $XY$ diagrams, the $b$ $XY$ rows are followed by $c^{\prime}$ rows with only a right square containing a $Y$. (For Type II diagrams, we must have $b \geq 1$.) Here, $c^{\prime}$ is an integer satisfying $1 \leq c^{\prime} \leq a$. Finally, a Type II diagram has $d^{\prime} \geq 0$ rows with only an $X$ in the right square.

\begin{figure}
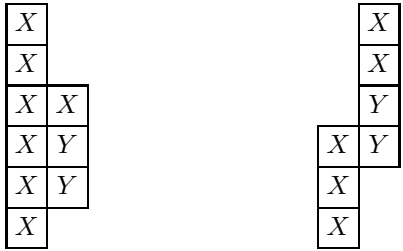

\begin{ytableau}
X & \none \\
X & \none \\
X & X \\
X & Y \\
X & Y \\
X & \none 
\end{ytableau}
\hspace{80pt}
\begin{ytableau}
\none & X \\
\none & X \\
\none & Y \\
X & Y \\
X & \none \\
X & \none 
\end{ytableau}
\caption{A Type I diagram on the left and a Type II diagram on the right. }
\label{fig:diagram-types}
\end{figure}

We would like to recover the area and $\hinv$ of the original labeled Dyck path $P$ from its $XY$ diagram. It is not hard to see that $\area(P)$ is equal to $a$, the number of rows at the bottom of the diagram containing only an $X$. The $\hinv$ of a diagram is equal to the number of pairs of left and right squares such that either
\begin{itemize} 
\item the left square appears immediately northwest of the right square, or
\item the two squares are in the same row and both contain $X$'s.
\end{itemize}

Now we can use the characterization given above to find the coefficient of \\ $s_{2^m, 1^{n-2m}}$ in $\risepoly_{n,1}(x;q,t)$ for any $0 \leq m \leq \lfloor n/2 \rfloor$. Since there are always at least as many $X$'s as $Y$'s in a diagram, we restrict our attention to diagrams with $m$ $Y$'s and $n-m$ $X$'s. Clearly the area of such a diagram may range between 0 and $n-m-1$, which corroborates the formula in Proposition \ref{prop:k=1-comb}. More precisely, we fix the area to be some value $0 \leq j \leq n-m-1$. If we can show that there is exactly one diagram with area $j$, $n-m$ $X$'s, and $m$ Y's with $\hinv = i$ for each $\max(0,m-j-1) \leq i \leq n-m-j-1$, then we have completed the proof of Proposition \ref{prop:k=1-comb}. 

First, we consider the possible Type I diagrams. We know that such a diagram must begin with $j$ rows consisting only of $X$'s in the left square followed by $m$ rows consisting of an $X$ in the left square and a $Y$ in the right square. Temporarily assuming $m \geq 1$, we have already accumulated $m-1$ $\hinv$.  We must place $n-2m-j$ more $X$'s. There are exactly $n-2m-j+1$ ways to accomplish this task. Namely, we choose any integer $0 \leq r \leq n-2m-j$. We repeatedly place an $X$ in the left square, then the right square, then the next left square up, and so on, placing $r$ $X$'s this way. After this, we stack the remaining $X$'s above the last of the $r$ $X$'s we had just placed. (If $r=0$, we place every $X$ in a stack above the highest $Y$.) We have created every Type I diagram with area $j$, $n-m$ $X$'s and $m$ $Y$'s. Furthermore, the resulting diagram has $\hinv = m-1+r$, so have contributed
\begin{align}
\label{contrib1}
\sum_{i=m-1}^{n-m-j-1} q^i 
\end{align}
to the coefficient of $t^j s_{2^m, 1^{n-2m}}$. If $m = 0$, the same logic shows that we have contributed
\begin{align}
\label{contrib2}
\sum_{i=0}^{n-j-1} q^i 
\end{align}
to the coefficient of $t^j s_{1^n}$.

Now we consider the Type II diagrams with area $j$, $n-m$ $X$'s, and $m$ $Y$'s. Such a diagram must begin with $j$ rows of just an $X$ in the left square, followed by $1 \leq b \leq m-1$ rows of an $X$ and a $Y$, contributing $b-1$ $\hinv$. Then the rest of the diagram is determined, as it must have $m-b$ rows that just have a $Y$ on the right followed by $n-m-j-b$ rows consisting of an $X$ on the right. Recall from the characterization of Type II diagrams that we must have $1 \leq m-b \leq j$, so actually $\max(1, m-j) \leq b \leq m-1$. This yields a contribution of 
\begin{align}
\label{contrib3}
\sum_{i=\max(0, m-j-1)}^{m-2} q^i
\end{align}
to the coefficient of $t^j s_{2^m, 1^{n-2m}}$. Gathering \eqref{contrib1}, \eqref{contrib2}, and \eqref{contrib3}, the coefficient of $t^j s_{2^m, 1^{n-2m}}$ in $\rise_{n,1}(x;q,t)$ is
\begin{align}
\sum_{i=\max(0,m-j-1)}^{n-m-j-1} q^i .
\end{align}
This concludes the proof of Proposition \ref{prop:k=1-comb}. 

\section{Extensions}
\label{sec:extensions}

\subsection{4-Variable Catalan polynomials}
\label{ssec:cat}

By the theory of shuffles, as described in Chapter 6 of \cite{haglund-book}, the Delta Conjecture implies the following conjecture.
\begin{align}
\left\langle \Delta^{\prime}_{e_k} e_n, e_n \right\rangle &=  \left. \sum_{D \in \dyck_{n}} q^{\dinv(D)} t^{\area(D)} \prod_{\areai_i(D) > \areai_{i-1}(D)} \left( 1 + z / t^{\areai_i(D)} \right)  \right|_{z^{n-k-1}}  \\
&= \left. \sum_{D \in \dyck_{n}} q^{\dinv(D)} t^{\area(D)} \prod_{i \in \val(D)} \left(1 + z / q^{\dinvi_{i}(D) + 1} \right)  \right|_{z^{n-k-1}} .
\end{align}
This conjecture is a generalization of the $q,t$-Catalan theorem proved by Garsia and the first author \cite{qtcatalan}.

Given the combinatorial interpretations in the Delta Conjecture, it is natural to wonder if we can combine them in a way that includes both products. Unfortunately, the polynomial
\begin{align}
\sum_{P \in \ldyck_{n}} q^{\dinv(P)} t^{\area(P)} \prod_{i \in \val(P)} \left( 1 + z / q^{\dinvi_i(P) + 1} \right) \prod_{\areai_i(P) > \areai_{i-1}(P)} \left( 1 + w / t^{\areai_i(P)} \right) x^P 
\end{align}
is not symmetric. However, we do seem to obtain an interesting polynomial in the Catalan case. We set
\begin{align}
\label{cat-def}
\cat_{n}(q,t,z,w) =& \sum_{D \in \dyck_{n}} q^{\dinv(D)} t^{\area(D)} \prod_{i \in \val(D)} \left( 1 + z / q^{\dinvi_i(D) + 1} \right) \\
\nonumber
&\times \prod_{\areai_i(D) > \areai_{i-1}(D)} \left( 1 + w / t^{\areai_i(D)} \right) .
\end{align}
As in the Delta Conjecture, we have a second (conjecturally equivalent) combinatorially defined polynomial. Given a labeled Dyck path $P$, let $h$ be the maximum area of any row in $P$. The \emph{reading order} processes the labels with area $h$ from right to left, then the labels with area $h-1$ from right to left, and so on until it has processed all labels.
We set $\bi_i(D)$ to be the number of diagonal inversions between the $i$th label in reading order and labels that precede it in reading order. We define
\begin{align}
\label{catmod-def}
\catmod_{n}(q,t,z,w) =& \sum_{D \in \dyck_{n}} q^{\dinv(D)} t^{\area(D)} \prod_{\bi_i(D) > \bi_{i-1}(D)} \left( 1 + z / q^{\bi_i(D)} \right) \\
\nonumber
&\times \prod_{\areai_i(D) > \areai_{i-1}(D)} \left( 1 + w / t^{\areai_i(D)} \right)  .
\end{align}
We note that $\cat_{n}(1,1,0,0) = \catmod_{n}(1,1,0,0) = \frac{1}{n+1} \binom{2n}{n}$, the usual Catalan number. We also have $\cat_n(1,1,1,1) = \catmod_n(1,1,1,1) = \frac{2^{n-1}}{n+1} \binom{2n}{n}$.  

Furthermore, we observe that $\bi_i(D) > \bi_{i-1}(D)$ if and only if the row $i$ contains a ``peak,'' a north step followed immediately by an east step, that is not the first peak in reading order. The number of these peaks for any Dyck path is equal to the number of valleys of the path. The ``zeta map'' or ``sweep map'' is a bijection $\dyck_n \to \dyck_n$ that interchanges rises and valleys and sends the joint distribution of $(\dinv, \area)$ to $(\area, \operatorname{bounce})$  \cite{haglund-book}. Hence, an application of the zeta map proves the following symmetries:
\begin{align}
\cat_n(1,1,z,w) &= \cat_n(1,1,w,z) = \catmod_n(1,1,z,w) = \catmod_n(1,1,w,z) \\
\cat_n(q,1,0,w) &= \cat_n(1,q,w,0)=\catmod_n(q,1,0,w) = \catmod_n(1,q,w,0).
\end{align}
We conjecture that the polynomials $\cat_n(q,t,z,w)$ and $\catmod_n(q,t,z,w)$ are equal and that they are connected to the delta operators. Since the first draft of this paper appeared, Mike Zabrocki proved all conjectures related to $\catmod_n(q,t,z,w)$; we describe Zabrocki's result in more detail at the end of this subsection.

\begin{conj}[4-Variable Catalan Conjecture]
\label{conj:cat}
\begin{align}
\label{cat1}
\left. \cat_{n}(q,t,z,w) \right|_{z^k w^{\el}} &= \left. \catmod_{n}(q,t,z,w) \right|_{z^k w^{\el}} \\
\label{cat2}
	&= \left\langle \Delta_{h_k } \nabla e_{n-k}, s_{\el+1,1^{n-k-\el-1}} \right\rangle  \\
\label{cat3}
	&= \left\langle \Delta_{h_k } \Delta^{\prime}_{e_{n-k-\el-1}} e_{n-k}, e_{n-k} \right\rangle .
\end{align}
Furthermore, each of these expressions is $k,\el$-symmetric.
\end{conj}

Later in this subsection we will show that \eqref{cat2} and \eqref{cat3} are equal.
We note that the equality of the right-hand sides of \eqref{cat1} and \eqref{cat2} would follow from the Rise Version of the Delta Conjecture. This is because taking the $z$ term in the product for $\catmod_n(q,t,z,w)$ corresponds to turning peaks into diagonal steps and using the $\dinv$ statistic for Schr\"oder paths \cite{schroder-conjecture, schroder}. It is also known that taking the inner product with a Schur function of hook shape yields the same Schr\"oder paths \cite{schroder-conjecture, schroder}. 

In fact, we can refine Conjecture \ref{conj:cat} based on how often the Dyck path returns to the diagonal. We define
\begin{align}
\cat_{n, r}(q, t, z, w) =&  \sum_{\substack{D \in \dyck_{n} \\ \areai_i(D) = 0 \text { $r$ times}}} q^{\dinv(D)} t^{\area(D)} \prod_{i \in \val(D)} \left( 1 + z / q^{\dinvi_i(D) + 1} \right) \\
\nonumber
&\times \prod_{\areai_i(D) > \areai_{i-1}(D)} \left( 1 + w / t^{\areai_i(D)} \right) \\
\catmod_{n, r}(q, t, z, w) =& \sum_{\substack{D \in \dyck_{n} \\ \areai_i(D) = 0 \text { $r$ times}}} q^{\dinv(D)} t^{\area(D)} \prod_{\bi_i(D) > \bi_{i-1}(D)} \left( 1 + z / q^{\bi_i(D)} \right) \\
\nonumber
&\times \prod_{\areai_i(D) > \areai_{i-1}(D)} \left( 1 + w / t^{\areai_i(D)} \right)  .
\end{align}

\begin{conj}[Touchpoint 4-Variable Catalan Conjecture]
\label{conj:touchpoint}
For integers $n \geq k$, $\el$, $r \geq 0$, we have
\begin{align}
\label{touchpoint1}
\left. \cat_{n,r}(q,t,z,w) \right|_{z^k w^{\el}} &= \left. \catmod_{n,r}(q,t,z,w) \right|_{z^k w^{\el}} \\
\label{touchpoint2}
&= \left\langle \Delta_{h_{\el}} \nabla E_{n-\el, r}, s_{k+1, 1^{n-k-\el-1}} \right\rangle \\
\label{touchpoint3}
&= \left\langle \Delta_{h_\el} \Delta^{\prime}_{e_{n-k-\el-1}} E_{n-\el, r}, e_{n-\el} \right\rangle .
\end{align}
where the polynomials $E_{n, r}$ are defined in \cite{schroder}.
\end{conj} 

Finally, we can refine part of Conjecture \ref{conj:touchpoint} based on exactly where the rows with $\areai_i = 0$ occur. Given a Dyck path $D$ in the sum for $\catmod_{n, m}(q, t, z, w)$, we know there are $r$ rows $i_1, i_2, \ldots, i_r$ such that $\areai_{h}(D) = 0$ if and only if $h$ is equal to some $i_j$. We decorate $D$ by placing a star next to (the second north step of) each double rise that corresponds to a power of $w$ that we select from the product
\begin{align}
\prod_{\areai_i(D) > \areai_{i-1}(D)} \left( 1 + w / t^{\areai_i(D)} \right) .
\end{align}
Then we form a composition $\alpha$ of $n - \el$ by setting
\begin{align}
\alpha_j = i_{j+1} - i_{j} - \text{\# of stars between rows $i_{j}$ and $i_{j+1}$} .
\end{align}
where $i_{r+1}$ is set equal to $n$. Set $\catmod_{n, \alpha}(q,t,z)$ to be the sum of the form of \eqref{catmod-def} over all Dyck paths decorated in this manner associated to composition $\alpha$. (We remove the variable $w$ from the notation because its power must equal $n - |\alpha|$.) The following conjecture is a refinement of Conjecture \ref{conj:touchpoint}, a fact which follows from work in \cite{compositional-shuffle}.

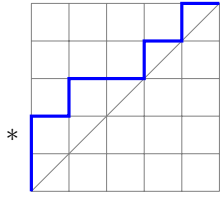
\begin{figure}
\centering
\begin{tikzpicture}
\draw[step=0.5cm, gray, very thin] (0.999,0) grid (3.5,2.5);
\draw[ gray, very thin] (1,0) -- (3.5,2.5);
\draw[blue, very thick] (1,0) -- (1,1) -- (1.5,1) -- (1.5,1.5) -- (2.5,1.5) -- (2.5, 2) -- (3, 2) -- (3,2.5) -- (3.5,2.5);
\node at (0.75, 0.75) {$\ast$};
\end{tikzpicture}
\caption{A decorated Dyck path with touch composition $\alpha = (2,1,1)$.}
\label{fig:decorated-dyck}
\end{figure}

\begin{conj}[Compositional 4-Variable Catalan Conjecture]
\label{conj:comp-cat}
For integers $n > k, \el \geq 0$ and a composition $\alpha \vDash n-\el$, we have
\begin{align}
\left. \catmod_{n, \alpha}(q,t,z) \right|_{z^k}
&= \left\langle \Delta_{h_{\el}} \nabla \cpoly_{\alpha}, s_{k+1, 1^{n-k-\el-1}} \right\rangle \\
&= \left\langle \Delta_{h_{\el}} \Delta^{\prime}_{e_{n-k-\el-1}} \cpoly_{\alpha}, e_{n-\el} \right\rangle .
\end{align}
where $\cpoly_{\alpha}$ is a certain symmetric function with coefficients in $\mathbb{Q}(q)$ which is defined in  \cite{compositional-shuffle}.
\end{conj}

All of this is quite different in the case of the Delta Conjecture, where replacing $e_n$ with $E_{n, m}$ or $\cpoly_{\alpha}$ does not necessarily yield a symmetric function whose coefficients are polynomials in $q$ and $t$. We note that the two symmetric function components in each of the above conjectures are equal. 

\begin{prop}
\label{prop:cat-sym}
For integers $m \geq k > 0$, a symmetric function $f \in \Lambda^{(m)}$, and any operator $\Gamma$ defined by $\Gamma \tilde{H}_{\mu} = g_{\mu} \tilde{H}_{\mu}$ for some $g_{\mu} \in \mathbb{Q}(q,t)$, we have
\begin{align}
\label{cat-sym}
\left\langle \Gamma \nabla f, s_{k+1, 1^{m-k-1}} \right\rangle &= \left\langle \Gamma \Delta^{\prime}_{e_{m-k-1}} f, e_{m} \right\rangle . 
\end{align} 
\end{prop}

\begin{proof}
The scalar product of ${\tilde H}_{\mu}$ with $s_{k+1,1^{m-k-1}}$ is  $e_{m-k-1} [B_{\mu}(q,t)-1]$, as proved on p.\ 362 of \cite{macdonald}. Since the scalar product of ${\tilde H}_{\mu}$ with $s_m$ is 1 for all $\mu \vdash m$, the left-hand side of our statement can be written as
\begin{align}
\langle \Gamma \Delta ^{\prime}_{e_{m-k-1}} \nabla f, s_m \rangle .
\end{align}
which can also be expressed as $\langle \Gamma \Delta ^{\prime} f, e_m \rangle$.
\end{proof}

Setting $m = n- \el$ and $\Gamma = \Delta_{h_\el}$ shows the desired equality of symmetric functions in the various conjectures in this section. 

Mike Zabrocki has recently proved Conjecture \ref{conj:comp-cat} \cite{zabrocki-catalan}. Zabrocki's result implies that the right-hand side of \eqref{touchpoint1} equals \eqref{touchpoint2} and \eqref{touchpoint3} and that the right-hand side of \eqref{cat1} equals \eqref{cat2} and \eqref{cat3}. The other equalities in our conjectures are still open.

\subsection{Conjectures for $\Delta_{h_\el} \Delta^{\prime}_{e_{n-k-1}} e_n$ and $\Delta_{h_\el} \nabla E_{n, r}$}

In this subsection we give combinatorial conjectures for the symmetric functions $\Delta_{h_\el} \Delta^{\prime}_{e_{n-k-1}} e_n$ and $\Delta_{h_\el} \nabla E_{n, r}$. We begin by defining our objects, which we call \emph{partially labeled Dyck paths} and denote $\pldyck_{n, \el}$. Given positive integers $n$ and $\el$, each element of $\pldyck_{n, \el}$ is a Dyck path of order $n+\el$ such that $n$ of its north steps are labeled with positive integers according to the following rules:
\begin{itemize}
\item if two labels share a column, the lower label is strictly smaller than the upper label (the usual rule for labeled Dyck paths), and
\item all of the north steps that do not receive a label are \emph{valleys}, i.e.\ they are north steps preceded by east steps.
\end{itemize}
In particular, the first north step must receive a nonzero label. Note that these objects cannot all be obtained by inserting ``empty'' valleys into labeled Dyck paths of order $n$, since this process forces a relationship between two labels separated by an empty valley (and we do not insist on any such relationship). We have drawn an example object in Figure \ref{fig:blank}.

\begin{figure}
\begin{center}
\begin{tikzpicture}
\draw[step=0.5cm, gray, very thin] (0,0) grid (4,4);
\draw[ gray, very thin] (0,0) -- (4,4);
\draw[blue, very thick] (0,0) -- (0,1) -- (0.5,1) -- (0.5,1.5) -- (1.5,1.5) -- (1.5, 2) -- (2, 2) -- (2,3.5) -- (3,3.5) -- (3,4) -- (4,4);

\node at (0.25,0.25) {2};
\node at (0.25,0.75) {4};
\node at (1.75,1.75) {3};
\node at (2.25,2.75) {1};
\node at (2.25,3.25) {5};
\node at (3.25,3.75) {6};

\end{tikzpicture}
\end{center}
\caption{An example $P \in \pldyck_{6, 2}$ with $\area(P) = 6$ and $\dinv(P) = 6$. The dinv occur in row pairs $(1, 4)$, $(2, 4)$, $(2, 5)$, $(2, 8)$, $(3, 6)$, and $(6, 8)$.}
\label{fig:blank}
\end{figure}
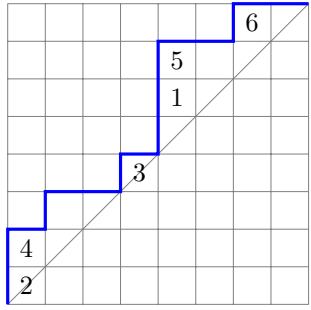

Given $P \in \pldyck_{n, \el}$, we will define statistics $\area(P)$ and $\dinv(P)$. $\area(P)$ is simply the area of the underlying Dyck path. To define $\dinv$, we simply label all the unlabeled north steps with 0's and then compute dinv as usual. That is, if $a$ is the vector that gives the area of each row and $\el_i$ is the label in row $i$ (which contains 0's in the formerly unlabeled rows), then we count the number of pairs $i < j$ such that
\begin{itemize}
\item $a_i = a_j$ and $\el_i < \el_j$, or
\item $a_i = a_j + 1$ and $\el_i > \el_j$.
\end{itemize}
Now we can make our first conjecture.

\begin{conj}
\label{conj:eh}
\begin{align}
\Delta_{h_\el} \Delta^{\prime}_{e_{n-k-1}} e_n &= \left. \sum_{P \in \pldyck_{n, \el}} q^{\dinv(P)} t^{\area(P)} x^P \prod_{i : a_{i}(P) > a_{i-1}(P)} \left(1 + z t^{-a_i(P)} \right) \right|_{z^k} .
\end{align}
\end{conj}

In the case where $k=0$, we can make a more refined conjecture involving the returns of $P$ to the diagonal. We say that $\touch(P)$ is equal to the number of rows $i$ with $a_i = 0$ that are not unlabeled valleys. For example, the object in Figure \ref{fig:blank} has $\touch(P) =  2$. 

\begin{conj}
\label{conj:eh-touch}
\begin{align}
\Delta_{h_\el} \nabla E_{n, r} &= \sum_{\substack{P \in \pldyck_{n, \el} \\ \touch(P) = r}} q^{\dinv(P)} t^{\area(P)} x^P .
\end{align}
\end{conj}

At this point, it is unclear how to refine these conjectures to allow for specific touch compositions.

It is worth noting that there is another dinv statistic that can replace the above definition. This dinv is more clearly related to the conjectures earlier in this section.
Given an entry $P \in \pldyck_{n,\el}$, we form a vector $d^{\prime}(P)$ defined by
\begin{itemize}
\item if row $i$ is empty, $d^{\prime}_i = -1$, else
\item $d^{\prime}$ is equal to the number of $j>i$ such that either 
\begin{itemize}
\item $a_i = a_j$ and row $j$ is empty,
\item $a_i = a_j + 1$ and row $j$ is empty, or
\item row $j$ is not empty and the two rows form a dinv pair by the usual definition for labeled Dyck paths.
\end{itemize}
\end{itemize}
Let $\dinv^{\prime}$ be the sum of the entries in the vector $d^{\prime}$. Computations suggest that we can replace $\dinv$ with $\dinv^{\prime}$ in the above conjectures. Then taking the scalar product with $e_n$ recovers special cases of the Catalan conjectures in Subsection \ref{ssec:cat}.

%

\section{Open problems}
\label{sec:final}

In this section, we describe a few open problems related to the Delta Conjecture. These accompany the problems of studying the $\minimaj$ statistic, mentioned at the end of Section \ref{sec:osp}, and of proving the extended conjectures described in Section \ref{sec:extensions}.

\subsection{Schr\"oder paths and 1,2-labeled Dyck paths}
In \cite{wilson-thesis}, the third author develops recursions for the polynomials  $\langle \risepoly_{n,k}(x;q,t), f \rangle$ and $\langle \valpoly_{n,k}(x;q,t), f \rangle$ for symmetric functions $f$ of the form $h_d e_{n-d}$ or $h_d h_{n-d}$. He then uses these recursions along with various results about $q$-binomial coefficients to resolve the Delta Conjecture in the case where we take the scalar product with $f$ on both sides and set $t=1/q$. It would be nice to remove the $t=1/q$ condition for this result, since this would yield a complete analog of the first author's results on the Shuffle Conjecture in \cite{schroder}. The main obstruction at this point is the symmetric function side.

\begin{prob}
Find recursions for the polynomials $\langle \Delta_{e_k} e_n, h_d e_{n-d} \rangle$ and \\ $\langle \Delta_{e_k} e_n, h_d h_{n-d} \rangle$ along the lines of the recursions obtained for the $k=n$ case in \cite{schroder}.
\end{prob}

\subsection{Symmetry of $\valpoly_{n,k}(x;q,t)$}
We mentioned in Subsection \ref{ssec:k=1-sym} that, due to the connection with LLT polynomials, we know that $\risepoly_{n,k}(x;q,t)$ is symmetric.  On the other hand, we have no such result for $\valpoly_{n,k}(x;q,t)$. In fact, we have observed that restricting the definition of $\valpoly_{n,k}(x;q,t)$ to labelings of a fixed Dyck path does \emph{not} always yield a symmetric function, which is in stark contrast to the $\risepoly_{n,k}(x;q,t)$ case. This implies that the following problem could be quite difficult. 

\begin{prob}
Prove $\valpoly_{n,k}(x;q,t)$ is a symmetric function, possibly by connecting it to (generalized?) LLT polynomials.
\end{prob}

The only partial results we have in this direction are that $\valpoly_{n,k}(x;q,0)$ (due to Proposition \ref{prop:omp}) and $\valpoly_{n,k}(x;1,t) = \risepoly_{n,k}(x;1,t)$ are symmetric. 

\subsection{Generalizations}
There are various ways one could generalize the Shuffle Conjecture. For example, we could replace $e_{k}$ by a general elementary symmetric function $e_{\lambda}$ or even a general symmetric function $f$. This would generalize the Fuss extension of the Shuffle Conjecture \cite{shuffle}. Alternatively, one could replace $e_n$ with $p_n$, which would hopefully have some relationship to the set of all lattice paths from $(0,0)$ to $(n,n)$, as developed for the Shuffle Conjecture in \cite{square}. We would also like to develop a concrete connection between the delta operator and the Rational Shuffle Conjectures of \cite{rational-shuffle} apart from the $t=1/q$ result mentioned in Section \ref{sec:t=1/q}. Finally, it would be quite interesting if one could find an extension of the module of diagonal harmonics with Frobenius characteristic equal to $\Delta_{e_k} e_n$. Brendon Rhoades, Mark Shimozono, and the first author are currently exploring a promising module for the one variable case (i.e.\ at $t=0$) which also generalizes the classical module of coinvariants of the symmetric group.

\subsection{Towards a proof of the Delta Conjecture}
Finally, we would be remiss if we did not mention the recent preprint of Carlsson and Mellit \cite{carlsson-mellit} which contains a proof of the Compositional Shuffle Conjecture, and therefore the $k=n-1$ case of the Delta Conjecture. It is quite possible that their proof could be adjusted to prove the Delta Conjecture, although it seems like this adjustment must be nontrivial. We are investigating ways to generalize the key recursions in \cite{carlsson-mellit} in order to apply them to the Rise Version of the Delta Conjecture. At the very least, this is a promising development towards a proof of our Delta Conjecture.


\appendix
\section{Completing the proof of Proposition \ref{prop:omp} }
\label{app:minimaj}

In this appendix, we prove the following statement, which appears as \eqref{omp-minimaj} in Proposition \ref{prop:omp}:
\begin{align}
\left. \valpoly_{n,k}(x;0,q) \right|_{M_{\alpha}}  &= \sum_{\pi \in \osp{\alpha}{k+1}} q^{\minimaj(\pi)} .
\end{align}
We will define a map
\begin{align}
\gamma_{\alpha, k} : \osp{\alpha}{k+1} \to \{ P \in \ldyck^{\dense}_{n,k} : \winv(P) = 0, x^P = \prod_{i=1}^{\ell(\alpha)} x_i^{\alpha_i} \} .
\end{align}
Then we will prove that this map is a bijection which satisfies $\area(\gamma_{\alpha,k}(\pi)) = \minimaj(\pi)$. Given $\pi \in \osp{\alpha}{k+1}$, we consider the permutation $\tau = \tau(\pi)$ as in the definition of $\minimaj$. Let $T$ be the positions of $\tau$ of the entries which are minimal in their blocks in $\pi$. We define the \emph{runs} of $\tau$ to be its maximal, contiguous, weakly increasing sequences. For convenience, we label the runs from right to left, saying that the rightmost run is the $0$th run. Say that $\tau$ has $s$ runs, and define positive integers $n = r_0 > r_1 > \ldots > r_s = 0$ such that the $i$th run of $\tau$ is equal to $\tau_{r_{i+1}+1}\ldots\tau_{r_{i}}$. Define $b^{i}_{1} < \ldots < b^{i}_{p_i}$ to be the positions of entries in the $i$th run of $\tau$ which are the leftmost entries in blocks which are entirely contained in the $i$th run of $\tau$. Finally, for each $i<s-1$, set $b^{i}_{0}$ to be the position of the leftmost entry in $\tau$ which shares a block with $\tau_{r_{s-i}}$. 

For example, set $\pi  = 13|23|14|234$ with $\tau = 312341234$. We decorate $\tau$ with bars after its minimal elements to obtain $31|23|41|234$. $\tau$ has 3 runs with $r_3 = 0$, $r_2=1$, $r_1=5$, and $r_4=9$. Using dashes to separate the runs, we get $3 - 1|23|4 -1|234$. We compute $b^0_1 = 7$, $b^0_0 = 5$, $b^1_1 = 3$ and $b^1_0 = 1$. Since the leftmost run does not contain any blocks, there are no $b^2_j$'s. 

We define $\gamma_{\alpha,k}(\pi)$ as follows. For $i = 0$ to $s-1$, we will insert the elements of the $i$th run of $\tau$ such that their rows in $P$ each have area equal to $i$. After each $i$, we will obtain a \emph{partial densely labeled Dyck path} $P^{(i+1)}$, which is densely labeled Dyck path whose set of labels does not necessarily form a composition. We begin with the empty densely labeled Dyck path $P^{(0)}$. To create $P^{(1)}$, we begin with the Dyck path $(NE)^{p_0}$. We label the squares from top to bottom with the sets $\tau_{b^{0}_{1}}\ldots\tau_{b^{0}_{2}-1}$, $\tau_{b^{0}_{2}}\ldots\tau_{b^{0}_{3}-1}$, \ldots, $\tau_{b^{0}_{p_{0}}}\ldots\tau_{n}$. Now we insert the entries $\tau_{b^0_0}\ldots\tau_{b^0_{1}-1}$ in a slightly more complicated fashion. We find the maximum entry in the northernmost square which is less than $\tau_{b^0_0}$; call this element $c$. By the definition of $\tau$, such a $c$ must exist. We insert a north step and then an east step immediately after the north step adjacent to this northernmost square. The new north square receives the label $\tau_{b^0_0}\ldots\tau_{r_1}$. The new east square's label contains $\tau_{r_1+1}\ldots\tau_{b^0_{1}-1}$ along with the entries in $c$'s square which are greater than $c$. In other words, we move these entries from $c$'s square to the new east square. The result is $P^{(1)}$. We can check $\winv(P^{(1)}) = 0$.

For greater values of $i$, we ``repeat'' this process as follows. We repeatedly insert $\tau_{b^i_{j}}\ldots\tau_{b^i_{j+1}-1}$ for $j = p_i$ down to $1$ just above the last east step added above. We leave the labels $\tau_{r_{i}+1}\ldots\tau_{b^{i-1}_{1}-1}$ in their east square and push the labels that were originally in $c$'s square that are greater than $c$ so that they are always in the highest east square with area equal to $i-1$. Then we find the maximum entry $c$ in the northernmost square with area $i$ such that $c < \tau_{b^i_0}$ and add new north and east squares as described above. The only remaining case to consider is if there is no $b^i_0$; then the new east squares label is just the entries in $c$'s square which are greater than $c$. We produce an example in Figure \ref{fig:gamma-map}.

\begin{figure}
\begin{tikzpicture}
\draw[step=0.6cm, gray, very thin] (-0.001,0) grid (0.6,0.6);
\draw[gray, very thin] (0,0) -- (0.6, 0.6);
\draw[blue, very thick] (0,0) -- (0,0.6) -- (0.6,0.6);
\node at (0.3, 0.3) {234};

\draw[step=0.6cm, gray, very thin] (1.199,0) grid (2.4, 1.2);
\draw[gray, very thin] (1.2,0) -- (2.4,1.2);
\draw[blue, very thick] (1.2,0) -- (1.2,1.2) -- (2.4,1.2);
\node at (1.5,0.3) {23};
\node at (1.5,0.9) {4};
\node at (2.1,0.9) {14};

\draw[step=0.6cm, gray, very thin] (2.999,0) grid (4.8,1.8);
\draw[gray, very thin] (3,0) -- (4.8,1.8);
\draw[blue, very thick] (3,0) -- (3,1.2) -- (3.6,1.2) -- (3.6,1.8) -- (4.8,1.8);
\node at (3.3,0.3) {23};
\node at (3.3,0.9) {4};
\node at (3.9,0.9) {1};
\node at (3.9,1.5) {23};
\node at (4.5,1.5) {4};

\draw[step=0.6cm, gray, very thin] (5.399,0) grid (7.8,2.4);
\draw[gray, very thin] (5.4,0) -- (7.8,2.4);
\draw[blue, very thick] (5.4,0) -- (5.4,1.2) -- (6,1.2) -- (6,2.4) -- (7.8,2.4);
\node at (5.7,0.3) {23};
\node at (5.7,0.9) {4};
\node at (6.3,0.9) {1};
\node at (6.3,1.5) {2};
\node at (6.3,2.1) {3};
\node at (6.9,2.1) {13};
\node at (7.5,2.1) {4};

\end{tikzpicture}
\caption{We compute $\phi_{(2,2,3,2),3}(13|23|14|234)$. From left to right, we depict $P^{(1)}$, $P^{(2)}$, $P^{(3)}$, and finally $P^{(4)}$.}
\label{fig:gamma-map}
\end{figure}
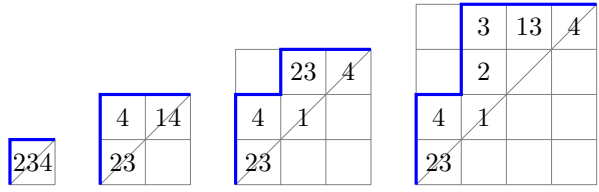

We note that, at each step, we have introduced zero $\winv$, so $\gamma_{\alpha,k}$ indeed maps to the paths
\begin{align}
\{ P \in \ldyck^{\dense}_{n,k} : \winv(P) = 0, x^P = \prod_{i=1}^{\ell(\alpha)} x_i^{\alpha_i} \} .
\end{align}
To see that $\gamma_{\alpha,k}$ is injective, we construct its inverse. We begin with the squares at maximum area in $P$. We remove them from top to bottom, using their labels to construct the blocks in the leftmost run in $\tau$. When we only have one square remaining at that area, we remove that square and form a block that consists of the labels in that square along with the smaller labels in the east square just to the right of that square (if there are any such labels). Then we move the larger labels into the north square below the square we just removed. We continue at the next largest area until all squares have been removed.

Next, we claim that $\gamma_{\alpha,k}$ is surjective. Carefully inspecting the image of $\gamma_{\alpha,k}$, we note that it contains any $P \in \ldyck^{\dense}_{n,k}$ with $\winv(P) = 0$ with the additional condition that every nonempty east square occurs either adjacent to the lowest two north squares with a given area or to the right of the uppermost north square at a given area. Essentially, if we see something of the form
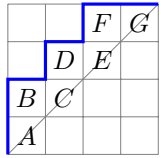
\begin{center}
\begin{tikzpicture}
\draw[step=0.5cm, gray, very thin] (-0.001,0) grid (2,2);
\draw[gray, very thin] (0,0) -- (2,2);
\draw[blue, very thick] (0,0) -- (0,1) -- (0.5,1) -- (0.5,1.5) -- (1,1.5) -- (1,2) -- (2,2);
\node at (0.25, 0.25) {$A$};
\node at (0.25, 0.75) {$B$};
\node at (0.75, 0.75) {$C$};
\node at (0.75, 1.25) {$D$};
\node at (1.25, 1.25) {$E$};
\node at (1.25, 1.75) {$F$};
\node at (1.75, 1.75) {$G$};
\end{tikzpicture}
\end{center}
then we must have $E = \emptyset$. It only remains to show that this condition is necessary in order to have $\winv(P)=0$. If $E \not = \emptyset$, it contains some element $e$. For $f = \min(F)$ and $d = \min(D)$, we must have $e < f \leq d$, since we have zero total $\winv$, so $e < d$. In order to have zero $\winv$, $e$ cannot be involved in any more diagonal inversions. However, either $e > a = \min(A)$ or $e \leq a < b = \min(B)$, so $e$ is involved in at least one more diagonal inversion, meaning that the total $\winv$ cannot be zero. Thus we must have $E = \emptyset$. 

Finally, we need to show that $\area(\gamma_{\alpha,k}(\pi)) = \minimaj(\pi)$. By definition, $\minimaj(\pi) = \maj(\tau)$, which is equivalent to the sum 
\begin{align}
\sum_{i=0}^{s-1} i (\# \text { of elements in run $i$ in $\tau$}).
\end{align}
Since each element of the $i$th run in $\tau$ is placed in a square with area $i$ in $\gamma_{\alpha,k}(\pi)$, we have $\area(\gamma_{\alpha,k}(\pi)) = \minimaj(\pi)$.

\bibliographystyle{alpha}
\bibliography{statistics}

\newcommand{\etalchar}[1]{$^{#1}$}
\begin{thebibliography}{HHL{\etalchar{+}}05b}

\bibitem[ALW14]{rational-catalan}
D.~Armstrong, N.~Loehr, and G.~S. Warrington.
\newblock Rational parking functions and {Catalan} numbers.
\newblock arXiv:1403.1845, March 2014.

\bibitem[BGLX15]{rational-shuffle}
F.~Bergeron, A.~Garsia, E.~S. Leven, and G.~Xin.
\newblock {Compositional} $(km, kn)$-{Shuffle} {Conjectures}.
\newblock {\em Int. Math. Research Notices}, October 2015.

\bibitem[CL95]{carre-leclerc}
C.~Carr\'e and B.~Leclerc.
\newblock Splitting the square of a {Schur function} into its symmetric and
  antisymmetric parts.
\newblock {\em J. Algebraic Combin.}, 4(3):201--231, 1995.

\bibitem[CM15]{carlsson-mellit}
E.~Carlsson and A.~Mellit.
\newblock A proof of the shuffle conjecture.
\newblock arXiv:math/1508.06239, August 2015.

\bibitem[EKKH03]{schroder-conjecture}
E.~Egge, D.~Kremer, K.~Killpatrick, and J.~Haglund.
\newblock A {Schr\"oder} generalization of {Haglund}'s statistic on {Catalan}
  paths.
\newblock {\em Electr. J. Combin.}, 10, 2003.
\newblock Research Paper 16, 21 pages (electronic).

\bibitem[GH03]{qtcatalan}
A.~Garsia and J.~Haglund.
\newblock A proof of the $q,t$-{Catalan} positivity conjecture.
\newblock {\em Advances in Math}, 175:319--334, 2003.

\bibitem[GLWX15]{rational-t=1/q}
A.~M. Garsia, E.~Leven, N.~Wallach, and G.~Xin.
\newblock A new plethystic symmetric function operator and the {Rational
  Compositional Shuffle Conjecture} at $t=1/q$.
\newblock arXiv:1501.00631, January 2015.

\bibitem[Hag04]{schroder}
J.~Haglund.
\newblock A proof of the $q,t$-{Schr\"oder} conjecture.
\newblock {\em Internat. Math. Res. Notices}, 11:525--560, 2004.

\bibitem[Hag08]{haglund-book}
J.~Haglund.
\newblock {\em The $q,t$-{Catalan} Numbers and the Space of Diagonal
  Harmonics}.
\newblock Amer. Math. Soc., 2008.
\newblock Vol. 41 of University Lecture Series.

\bibitem[Hai02]{delta}
M.~Haiman.
\newblock Vanishing theorems and character formulas for the {Hilbert} scheme of
  points in the plane.
\newblock {\em Invent. Math.}, 149:371--407, 2002.

\bibitem[HHL05a]{hhl}
J.~Haglund, M.~Haiman, and N.~Loehr.
\newblock A combinatorial formula for {Macdonald} polynomials.
\newblock {\em J. Amer. Math. Soc.}, 18:735--761, 2005.

\bibitem[HHL{\etalchar{+}}05b]{shuffle}
J.~Haglund, M.~Haiman, N.~Loehr, J.~B. Remmel, and A.~Ulyanov.
\newblock A combinatorial formula for the character of the diagonal
  coinvariants.
\newblock {\em Duke Math. J.}, 126:195--232, 2005.

\bibitem[HMZ12]{compositional-shuffle}
J.~Haglund, J.~Morse, and M.~Zabrocki.
\newblock A compositional shuffle conjecture specifying touch points of the
  {Dyck} path.
\newblock {\em Canad. J. of Math.}, 64:822--844, 2012.

\bibitem[LLT97]{llt}
A.~Lascoux, B.~Leclerc, and J.-Y. Thibon.
\newblock Ribbon tableaux, {Hall-Littlewood} functions, quantum affine
  algebras, and unipotent varieties.
\newblock {\em J. Math. Phys.}, 38(2):1041--1068, 1997.

\bibitem[LW07]{square}
N.~A. Loehr and G.~S. Warrington.
\newblock Square $q, t$-lattice paths and $\nabla(p_n)$.
\newblock {\em Trans. Amer. Math. Soc.}, 359(2):649--669, 2007.

\bibitem[Mac95]{macdonald}
I.~Macdonald.
\newblock {\em Symmetric Functions and Hall Polynomials}.
\newblock Oxford University Press, second edition, 1995.

\bibitem[Oli65]{q-lucas}
G.~Olive.
\newblock Generalized powers.
\newblock {\em Amer. Math. Monthly}, 72:619--627, 1965.

\bibitem[Rho16]{rhoades-minimaj}
Brendon Rhoades.
\newblock Ordered set partition statistics and the {Delta Conjecture}.
\newblock arXiv:105.04007, May 2016.

\bibitem[RW15]{rw}
J.~B. Remmel and A.~T. Wilson.
\newblock An extension of {MacMahon's} equidistribution theorem to ordered set
  partitions.
\newblock {\em J. Combin. Theory, Ser. A}, 134:242--277, August 2015.

\bibitem[Sag92]{sagan-congruence}
B.~E. Sagan.
\newblock Congruence properties of $q$-analogs.
\newblock {\em Advances in Math.}, 95:127--143, 1992.

\bibitem[Sta99]{ec2}
R.~P. Stanley.
\newblock {\em Enumerative Combinatorics}, volume~2.
\newblock Cambridge University Press, 1999.

\bibitem[vL00]{van-leeuwen}
M.~A.~A. van Leeuwen.
\newblock Some bijective correspondences involving domino tableaux.
\newblock {\em Electron. J. Combin.}, 7:25, 2000.
\newblock Research Paper 35.

\bibitem[Wil15]{wilson-thesis}
A.~T. Wilson.
\newblock {\em Generalized shuffle conjectures for the Garsia-Haiman delta
  operator}.
\newblock PhD thesis, UCSD, 2015.

\bibitem[Wil16]{tesler-wilson}
A.~T. Wilson.
\newblock A weighted sum over generalized {Tesler} matrices.
\newblock {\em J. Algebraic Combin.}, pages 1--31, 2016.
\newblock doi:10.1007/s10801-016-0726-2.

\bibitem[Zab16]{zabrocki-catalan}
M.~Zabrocki.
\newblock A proof of the 4-variable {Catalan} polynomial of the {Delta}
  conjecture.
\newblock arXiv:1609.03497, September 2016.

\end{thebibliography}
\label{sec:biblio}

\end{document}